\documentclass[12pt]{article}
\usepackage{amssymb}
\usepackage{amscd}
\usepackage{amsfonts}
\usepackage{times}
\usepackage{mathptmx}
\usepackage{amsmath}
\usepackage[usenames]{color}
\usepackage{mathrsfs}
\usepackage{amsfonts}
\usepackage{amssymb,amsmath}
\usepackage{CJK}
\usepackage{cite}
\usepackage{cases}
\usepackage{amsthm}
\usepackage{extarrows}
\usepackage{indentfirst,latexsym,bm}
\usepackage{amsthm}
\def\cl{\centerline}
\def\la{\lambda}
\def\a{\alpha}

\def\b{\beta}
\def\pa{\partial}
\def\vs{\vspace*}

\def\LL{\mathscr{L}}

\def\Z{\mathbb{Z}}

\def\C{\mathbb{C}}

\def\Vir{\hbox{Vir}}

\def\vs{\vspace*}

\numberwithin{equation}{section}
\newtheorem{theo}{Theorem}[section]
\newtheorem{defi}[theo]{Definition}
\newtheorem{coro}[theo]{Corollary}
\newtheorem{lemm}[theo]{Lemma}
\newtheorem{prop}[theo]{Proposition}

\newtheorem{rem}[theo]{Remark}

\begin{document}
\begin{center}
{\bf\large Extensions of Schr\"odinger-Virasoro conformal modules}
\footnote {$^{\,*}$Corresponding author: lmyuan@hit.edu.cn (Lamei Yuan)}
\end{center}

\cl{Lamei Yuan$^{\,*}$, Kaijing Ling}

\cl{\small Department of Mathematics, Harbin Institute of Technology, Harbin
150001, China}

\cl{\small\footnotesize E-mails: lmyuan@hit.edu.cn,
kjling\_edu@126.com }
\vs{8pt}

{\small
\parskip .005 truein
\baselineskip 3pt \lineskip 3pt
\noindent{{\bf Abstract:} In this paper, we study extensions between two finite irreducible conformal modules over the Schr\"odinger-Virasoro conformal algebra and the extended Schr\"odinger-Virasoro conformal algebra. Also, we classify all finite nontrivial irreducible conformal modules over the extended Schr\"odinger-Virasoro conformal algebra. As a byproduct, we obtain a classification of extensions of Heisenberg-Virasoro conformal modules.

 \vs{5pt}


\vs{5pt}

\noindent{\bf MR(2000) Subject Classification:}~ 17B10, 17B65, 17B68

\parskip .001 truein\baselineskip 6pt \lineskip 6pt

\section{Introduction}
\vs{8pt}

Lie conformal algebras (LCAs), introduced by V. Kac, represent an axiomatic description of operator product expansion (OPE) in conformal field theory. In particular, $\la$-brackets arise as generating functions for the singular part of the OPE. The structure, cohomology and representation theory of LCAs was developed by V. Kac and his coworkers in the late 1990s (\cite{BKV,CK,DK,CKW1,CKW2,Kac1}), and non-semisimple LCAs associated to infinite-dimensional Lie algebras of Virasoro type were studied recently in \cite{SY1,SY2,SXY,WY,YW1,YW2}. As pointed out in \cite{CK}, conformal modules of LCAs are not completely reducible in general. Therefore it is necessary to study the extension problem. Extensions between two finite irreducible conformal modules over the Virasoro, the current and the Neveu-Schwarz and the semi-direct sum of the Virasoro and the current conformal algebras were classified by S.-J. Cheng, V. Kac and M. Wakimoto in \cite{CKW1,CKW2}. By using their techniques, Ngau Lam solved the extension problem for the supercurrent conformal algebras in \cite{L}.

In this paper, we aim to study extensions of modules over the Schr\"odinger-Virasoro conformal algebra and the extended Schr\"odinger-Virasoro conformal algebra, which were introduced in \cite{SY1} as Lie conformal algebras associated to the Schr\"odinger-Virasoro Lie algebra and the extended Schr\"odinger-Virasoro Lie algebra, respectively.
The Schr\"odinger-Virasoro conformal algebra is defined as a finite free Lie conformal algebra
$\mathrm{SV}=\mathbb{C}[\partial]L\oplus
\mathbb{C}[\partial]M\oplus\mathbb{C}[\partial]Y$ endowed with the
following  nontrivial $\lambda$-brackets \vspace*{-7pt}
\begin{eqnarray}\label{lamda-bracket}
[L_\lambda L]=(\partial+2\lambda)L,\ \ {[L_\lambda Y]}=(\partial+\frac32\lambda)Y,\ {[L_\lambda M]}=(\partial+\lambda)M,\ \,{[Y_\lambda
Y]}=(\partial+2\lambda)M,
\end{eqnarray}
whereas the extended Schr\"odinger-Virasoro conformal algebra
is a finite free Lie conformal algebra $\widetilde
{\mathrm{SV}}=\mathbb{C}[\partial]L\oplus
\mathbb{C}[\partial]M\oplus\mathbb{C}[\partial]Y\oplus\mathbb{C}[\partial]N$
endowed with the nontrivial $\lambda$-brackets defined by
(\ref{lamda-bracket}), together with the
following nontrivial ones
\begin{eqnarray}\label{lamda-bracket1}
[L_\lambda N]=(\partial+\lambda)N,\ \ [N_\lambda M]=2M,\ \ [N_\lambda Y]=Y.
\end{eqnarray}
Note that both $\mathrm{SV}$ and $\widetilde {\mathrm{SV}}$ are non-semisimple LCAs, and they contain the
Virasoro conformal algebra $\Vir$ as a subalgebra, which is a
free $\C[\partial]$-module generated by $L$ such that
\begin{eqnarray}\label{Vir}
\Vir=\C[\partial]L,\ \ \ [L_\lambda L]=(\partial+2\lambda)L.
\end{eqnarray}
Moreover, the extended Schr\"odinger-Virasoro conformal algebra contains the Heisenberg-Virasoro
conformal algebra as a subalgebra, which is generated by $L$ and $N$ as a $\C[\pa]$-module and satisfies
 the following $\lambda$-brackets
\begin{eqnarray}\label{elie}
[L_\lambda L]=(\partial+2\lambda)L,\ \ [L_\lambda
N]=(\partial+\lambda)N,\ \ [N_\lambda L]=\lambda N,\ \ [N_\lambda
N]=0.
\end{eqnarray}
Therefore, the classification of extensions of Virasoro conformal modules obtained in \cite{CKW1} will be used in our study, and one will see that our result can be directly applied to the Heisenberg-Virasoro
conformal algebra.

The paper is organized as follows. In Section 2, we first review the notion of a Lie conformal algebra, and that of an extended annihilation algebra. The definition of conformal modules and their extensions over a Lie conformal algebra will be also recalled. Then we list some known results, including the
classification of finite nontrivial irreducible
conformal modules over the Schr\"odinger-Virasoro conformal algebra (see Proposition \ref{p1}) and the
classification of extensions
of modules over the Virasoro conformal algebra (see Theorems \ref{th2}-\ref{th4}). These results will be important for the rest of the paper.
In Section 3, we
study extensions between two finite irreducible conformal modules over the Schr\"odinger-Virasoro conformal algebra. Three types of extensions of Schr\"odinger-Virasoro conformal modules will be discussed. Explicitly, there is a one-dimensional conformal module involved in the first two types, whereas  both modules involved in the third type are non-one-dimensional.
In Section 4, we classify and construct all finite irreducible conformal modules over the extended Schr\"odinger-Virasoro conformal algebra by using the equivalent language of LCAs and extended annihilation algebras, and some techniques developed in \cite{CK,WY}. Section 5 is then devoted to a classification of extensions of the extended Schr\"odinger-Virasoro conformal modules, and the results will be applied to the Heisenberg-Virasoro
conformal algebra in the last section.

Throughout the paper, all vector spaces, tensor products, and algebras are assumed to be over the field of complex numbers $\C$. In addition to the standard notation $\Z$, we use $\Z^+$ (resp. $\C^*$)  to denote the set of nonnegative integers (resp. nonzero complex numbers).
\section{Preliminaries}

In this section, we recall the definition of Lie conformal algebras, annihilation algebras, conformal modules and their extensions, and some known results that we need in this paper. For more details, the reader is referred to \cite{CK,CKW1,Kac1,WY}.

\begin{defi}\rm
A Lie conformal algebra $R$ is a $\C[\partial]$-module endowed with a $\C$-bilinear map
$$ R\otimes R\rightarrow \C[\lambda]\otimes R,\ \  a\otimes b \mapsto [a_\lambda b],$$
called the $\la$-bracket, and
satisfying the following axioms ($a, b, c\in R$),
\begin{eqnarray}
[\partial a_\lambda b]&=&-\lambda[a_\lambda b],\ \ [ a_\lambda \partial b]=(\partial+\lambda)[a_\lambda b] \ \ \mbox{(conformal\  sesquilinearity)},\label{Lc1}\\
{[a_\lambda b]} &=& -[b_{-\lambda-\partial}a] \ \ \mbox{(skew-symmetry)},\label{Lc2}\\
{[a_\lambda[b_\mu c]]}&=&[[a_\lambda b]_{\lambda+\mu
}c]+[b_\mu[a_\lambda c]]\ \ \mbox{(Jacobi \ identity)}\label{Lc3}.
\end{eqnarray}
\end{defi}

Let $R$ be a Lie conformal algebra.
For each $j\in\Z^+$, we can define the {\it $j$-product} $a_{(j)}b$ of any two elements $a,b\in R$ by the following generating series:
\begin{equation}\label{jj1}
[a_{\lambda}b]=\sum_{j\in\Z^{+}}(a_{(j)}b)\frac{\lambda^j}{j!}.
\end{equation}
Then the following axioms of $j$-products hold:
\begin{equation}\label{lcaj}
\aligned
&a_{(n)}b=0,\ {\rm for}\ n\gg0;\\
&(\partial a)_{(n)}b=-na_{(n-1)}b;\\
&a_{(n)}b=\sum_{j\in\Z^{+}}(-1)^{n+j+1}\frac{\partial^j}{j!} b_{(n+j)}a;\\
&[a_{(m)},b_{(n)}]=\sum_{j=0}^m \begin{pmatrix}
m\\j
\end{pmatrix}(a_{(j)}b)_{(m+n-j)}.
\endaligned
\end{equation}
Actually, one can also define Lie conformal algebras using the language of $j$-products (c.f. \cite{Kac1}).

Consider the space $\widetilde{R}=R\bigotimes\C[t,t^{-1}]$ with $\widetilde{\partial}=\partial\otimes{\rm id}+{\rm id}\otimes\partial_t$, where $\rm id$ appearing on the left
(resp. right) of $\bigotimes$ is the identity operator acting on $R$ (resp. $\C[t,t^{-1}]$). This space is called the {\it affinization}
of $R$. Its generating elements can be written as $a\otimes t^m$, where $a\in R$ and $m\in\Z$. For clarity, we will use the notation $\widetilde{R}=R[t,t^{-1}]$, $a t^m$ for its elements and $\widetilde{\partial}=\partial+\partial_t$. By \eqref{lcaj}, we obtain a well defined commutation relation on $\widetilde{R}$:
\begin{equation}\label{ppp2}
[at^m, bt^n]=\sum_{j\in\Z^{+}}
\begin{pmatrix}
m\\j
\end{pmatrix}
(a_{(j)}b)t^{m+n-j}, \ \forall\ at^m,\, bt^n\in \widetilde{R},
\end{equation}
which gives $\widetilde{R}$ a structure of algebra, denoted by $( \widetilde{R}, [\cdot,\cdot])$.
It can be verified that the subspace $\widetilde{\partial}\widetilde{R}$ spanned by elements of the form $\{(\pa a) t^n+n a t^{n-1}| n\in\Z\}$ is a  two-sided ideal of the algebra $( \widetilde{R}, [\cdot,\cdot])$. Set
\begin{equation}
Lie(R)={\widetilde{R}}/{\widetilde{\partial}\widetilde{R}}.
\end{equation}
Let $a_{(m)}$ denote the image of $at^m$ in $Lie(R)$. Then $(\pa a)_{(n)}=-n a_{(n-1)}$. Define a bracket on $Lie(R)$ by
\begin{eqnarray}\label{lie}
[a_{(m)},b_{(n)}]=\sum_{j=0}^m \begin{pmatrix}
m\\j
\end{pmatrix}(a_{(j)}b)_{(m+n-j)},
\end{eqnarray}
for $a,b\in R$, $m, n\in\Z$. One can check that $(Lie(R),[\cdot,\cdot])$ is a Lie algebra with respect to \eqref{lie}. Note that $Lie(R)$ admits a derivation $\pa$ defined by $\pa(a_{(n)})=-n a_{(n-1)}$, for $a\in R$ and $n\in\Z$.
The Lie subalgebra $\textit{Lie}(R)^+={\rm span}_{\C}\{a_{(m)}\,|\,a\in R, m\in\Z^+\}$ is called the {\it annihilation algebra} of $R$.
The {\it extended annihilation algebra} $\textit{Lie}(R)^e$ is defined as the
semidirect product of the $1$-dimensional Lie subalgebra $\C \pa$ and $\textit{Lie}(R)^+$ with the action
$[\pa,a_{(n)}]=-na_{(n-1)}$.

\begin{defi}\label{def1} \rm A {\it conformal module} $V$ over a Lie conformal algebra $R$
is a $\mathbb{C}[\partial]$-module endowed with a $\C$-linear map
$R\rightarrow {{\rm End_{\mathbb{C}}}V \bigotimes_{\mathbb{C}}\mathbb{C}[\lambda]}$, $a\mapsto a_\lambda $, satisfying the following conditions for all $a,b\in R$:
\begin{eqnarray*}
[a_\lambda,b_\mu]=a_\lambda b_\mu -b_\mu a_\lambda =[a_\lambda b]_{\lambda+\mu},\
(\partial a)_\lambda =[\partial,a_\lambda]=-\lambda a_\lambda.
\end{eqnarray*}
A conformal module $V$ over a Lie conformal algebra $R$ is called {\it finite}
if $V$ is finitely generated over
$\mathbb{C}[\partial]$. A conformal module $V$ over a Lie conformal algebra $R$ is called {\it irreducible} if there is no nontrivial invariant subspace.
\end{defi}

Let $V$ be a conformal module over a Lie conformal algebra $R$.
An element $v$ in $V$ is called an {\it invariant} if $R_\la v=0$.
Denote by $V^0$ the subspace of invariants of $V$.
It is easy to see that $V^0$ is a conformal submodule of $V$.
If $V^0=V$, then $V$ is called a {\it trivial} $R$-module.
A trivial module just admits the structure of a $\C[\partial]$-module. The vector space $\C$ is viewed as a trivial module with trivial actions of both $\pa$ and $R$.
For any fixed complex constant $\alpha$, there
is a natural trivial $R$-module $\C c_\alpha$, such that $\C c_\alpha=\C$ and $\pa v=\alpha v,\ R_\la v=0$ for $v\in\C c_\alpha$.
The modules
$\C c_\alpha$ (with $\alpha\in\C$) exhaust all trivial irreducible $R$-modules.

The following result is due to \cite[Lemma 2.2]{Kac2}.\vs{-8pt}

\begin{lemm}
Let $R$ be a Lie conformal algebra and $V$ an $R$-module.\\
1) If $\partial v= a v$ for some $a\in \C$ and $v\in V$, then $R_\la v=0$.\\
2) If $V$ is a finite conformal module without any nonzero invariants, then $V$ is a free $\C[\partial]$-module.
\end{lemm}

An element $v\in V$ is called a {\it torsion element} if there exists a nonzero polynomial $p(\partial)\in\C[\partial]$ such that $p(\partial)v=0$.
For any $\C[\partial]$-module $V$, there exists a nonzero torsion element if and only if there exists a nonzero $v\in V$
such that $\partial v=a v$ for some $a\in\C$.
A finitely generated $\C[\partial]$-module $V$ is free if and only if $0$ is the only torsion element of $V$.

We deduce from the above discussions the following important result.

\begin{lemm}\label{I}
Let $R$ be a Lie conformal algebra.
Then every finite nontrivial irreducible $R$-module has no nonzero torsion element and
is a free $\C[\partial]$-module.
\end{lemm}


Assume that $V$ is a conformal module over $R$.
We can also define {\it $j$-actions} of $R$ on $V$ using the following generating series
\begin{equation}
a_{\lambda}v=\sum_{j\in\Z^{+}}(a_{(j)}v)\frac{\lambda^j}{j!}.
\end{equation}
The $j$-actions satisfy relations similar to those in (\ref{lcaj}).

It is immediate to see that a conformal module $V$ over a Lie conformal algebra $R$ is the same as a module over the extended annihilation algebra $\textit{Lie}(R)^e$ satisfying the local nilpotent condition
\begin{equation}\label{conformal}
a_{(n)}v=0,\ \ for\ \ v\in V,\ a\in R, \ n\gg 0.
\end{equation}
A $\textit{Lie}(R)^e$-module satisfying this condition is called {\it conformal}.

\begin{defi}\label{def2}
Let $V$ and $W$ be two modules over a Lie conformal algebra (or a Lie algebra) $R$. An {\it extension} of $W$ by $V$ is an exact sequence of $R$-modules of the form
\begin{eqnarray}\label{Em}
0\longrightarrow V\xlongrightarrow{i} E \xlongrightarrow{p} W \longrightarrow 0.
\end{eqnarray}
Two extensions $0\longrightarrow V\xlongrightarrow{i} E \xlongrightarrow{p} W \longrightarrow 0$ and $0\longrightarrow V\xlongrightarrow{i'} E' \xlongrightarrow{p'} W \longrightarrow 0$ are said to be {\it equivalent} if there exists a commutative diagram of the form
\begin{equation*}
\begin{CD}
0@>>> V @>i>{\rm }>  E @>p>> W@>>> 0\\
@. @V{1_V}VV @V\psi VV @V{1_W}VV\\
0@>>> V @>i'>{\rm }> E' @>p'>>
W @>>> 0,
\end{CD}
\end{equation*}
where $1_V: V\rightarrow V$ and $1_W: W\rightarrow W$ are the respective identity maps and $\psi: E\rightarrow E'$ is a homomorphism of modules.
\end{defi}

The direct sum of modules $V\bigoplus W$ obviously gives rise to an extension. Extensions equivalent to it are called {\it trivial extensions}. In general, an extension can be thought of as the direct sum of vector spaces $E=V\oplus W$, where $V$ is a submodule of $E$, while for $w$ in $W$ we have:
\begin{equation*}
a\cdot w=aw+\phi_a(w),\ \ a\in R,
\end{equation*}
where $\phi_a:W\rightarrow V$ is a linear map satisfying the cocycle condition: $$\phi_{[a,b]}(w)=\phi_a(bw)+a\phi_b(w)-\phi_b(aw)-b\phi_a(w),\ \, b\in R.$$ The set of these cocycles form a vector space over $\C$. Cocycles equivalent to the trivial extension are called {\it coboundaries}. They form a subspace and the quotient space by it is denoted by $\textrm{Ext}(W, V).$

It was shown in
\cite{CK} that all free nontrivial Virasoro conformal modules of rank one
over $\mathbb{C}[\partial]$ are the following ones $(\Delta,
\alpha\in \mathbb{C})$:
\begin{eqnarray}
M(\alpha,\Delta)=\mathbb{C}[\partial]v_\Delta,\ \ L_\lambda
v=(\partial+\alpha+\Delta \lambda)v_\Delta.
\end{eqnarray}
The module $M(\alpha,\Delta)$ is irreducible if and only if
$\Delta\neq 0$. The module $M(\alpha,0)$ contains a unique
nontrivial submodule $(\partial +\alpha)M(\alpha,0)$ isomorphic to
$M(\alpha,1).$  Moreover, the modules $M(\alpha,\Delta)$ with
$\Delta\neq 0$ exhaust all finite non-1-dimensional irreducible
Virasoro conformal modules. Therefore $M(\alpha,\Delta)$ with $\Delta\neq 0$, together with the one-dimensional modules $\C c_\beta$ ($\beta\in\C$), form a complete list of finite irreducible conformal modules over the Virasoro conformal algebra.

In \cite{CKW1}, extensions over the finite irreducible Virasoro conformal modules of the following three types have been classified ($\Delta, \bar\Delta\in\C^*$):
\begin{eqnarray}
&&0\longrightarrow \C{c_\gamma}\longrightarrow E \longrightarrow M(\alpha,\Delta) \longrightarrow 0,\label{type1}\\
&&0\longrightarrow M(\alpha,\Delta)\longrightarrow E \longrightarrow \C c_\gamma \longrightarrow 0,\label{type2}\\
&&0\longrightarrow M(\bar\alpha,\bar\Delta)\longrightarrow E \longrightarrow M(\alpha,\Delta) \longrightarrow 0. \label{type3}
\end{eqnarray}
We list the corresponding results in the following three theorems for later use.

\begin{theo}\label{th2}
Nontrivial extensions of the form \eqref{type1}\ exist if and only if  $\alpha+\gamma=0$ and $\Delta=1$ or $2$. In these cases, they are given (up to equivalence) by $$L_\lambda v_\Delta=(\partial+\alpha+\Delta\lambda)v_\Delta+f(\lambda)c_\gamma,$$ where
\begin{itemize}
\item[{\rm (i)}] $f(\lambda)=c_2\lambda^2$, for $\Delta=1$ and $c_2\neq0$.
\item [{\rm (ii)}] $f(\lambda)=c_3\lambda^3$, for $\Delta=2$ and $c_3\neq0$.
\end{itemize}
Furthermore, all trivial cocycles are given by scalar multiples of the polynomial $f(\lambda)=\alpha+\gamma+\Delta\lambda$.
\end{theo}

\begin{theo}\label{th3}
Nontrivial extensions of Virasoro conformal modules of the form \eqref{type2} exist if and only if $\alpha+\gamma=0$ and $\Delta=1$. These extensions are given (up to equivalence) by

\begin{eqnarray*}
L_\lambda c_\gamma=f(\partial,\lambda)v_\Delta,\ \
\partial c_\gamma=\gamma c_\gamma+h(\partial)v_\Delta,
\end{eqnarray*}
where $f(\partial,\lambda)=h(\partial)=a_0\in\C^*$.
\end{theo}

\begin{theo}\label{th4} Nontrivial extensions of Virasoro conformal modules of the form \eqref{type3} exist only if $\alpha=\bar\alpha$ and $\Delta-\bar\Delta=0,2,3,4,5,6.$ In these cases, they are given (up to equivalence) by $$L_\lambda v_\Delta=(\partial+\alpha+\Delta\lambda)v_\Delta+f(\partial,\lambda)v_{\bar\Delta },$$
where the values of $\Delta$ and $\bar\Delta$ along with the corresponding polynomials $f(\pa,\la)$ whose nonzero scalar multiples give rise to nontrivial extensions are listed as follows ( $\bar\partial=\partial+\alpha$):
\begin{itemize}
\item[{\rm (i)}] $\Delta=\bar\Delta, f(\pa,\la)=c_0+c_1\la, (c_0,c_1)\neq(0,0).$
\item [{\rm (ii)}]$\Delta-\bar\Delta=2, f(\pa,\la)=\la^2(2\bar\pa+\la)$.
\item [{\rm (iii)}]$\Delta-\bar\Delta=3, f(\pa,\la)=\bar\pa\la^2(\bar\pa+\la)$.
\item [{\rm (iv)}]$\Delta-\bar\Delta=4, f(\pa,\la)=\la^2(4\bar\pa^3+6\bar\pa^2\la-\bar\pa\la^2+\bar\Delta\la^3)$.
\item [{\rm (v)}]$(\Delta,\bar\Delta)=(1,-4), f(\pa,\la)=\bar\pa^4\la^2-10\bar\pa^2\la^4-17\bar\pa\la^5-8\la^6$.
\item [{\rm (vi)}]$(\Delta,\bar\Delta)=(\frac72\pm\frac{\sqrt{19}}2, -\frac52\pm\frac{\sqrt{19}}2), f(\pa,\la)=\bar\pa^4\la^3-(2\bar\Delta+3)\bar\pa^3\la^4-3\bar\Delta\bar\pa^2\la^5-(3\bar\Delta+1)\bar\pa\la^6-(\bar\Delta+\frac9{28})\la^7$.
\end{itemize}
\end{theo}

\begin{rem} {\rm We rewrite \cite[Theorem 2.4]{CKW2} (see also \cite[Theorem 3.2]{CKW1}) as Theorem \ref{th4} by removing the cases $\Delta=0$ or $\bar\Delta=0$. The reason is that both $M(\alpha,\Delta)$ and $M(\bar\alpha,\bar\Delta)$ in \eqref{type3} are irreducible Virasoro conformal modules.}
\end{rem}

All finite nontrivial irreducible conformal modules over the Schr\"odinger-Virasoro conformal algebra were classified in \cite{WY}, and the corresponding results are the following.

\begin{prop} \label{p1} Any finite nontrivial irreducible conformal
 module over the Schr\"odinger-Virasoro conformal algebra
is of the form
\begin{eqnarray*}
V(\alpha,\Delta)=\mathbb{C}[\partial]v_\Delta,\ L_\lambda
v_\Delta=(\partial+\alpha+\Delta \lambda)v_\Delta, \ M_\lambda v_\Delta=Y_\lambda v_\Delta=0,
\end{eqnarray*}
where $\alpha,\,\Delta \in\C$ and $\Delta\neq 0$.
\end{prop}

\begin{rem} {\rm We see from Proposition \ref{p1} that a finite nontrivial irreducible Schr\"odinger-Virasoro conformal module just admits the structure of a Virasoro conformal module. However, one will see in Section 3 that there are much more nontrivial extensions of Schr\"odinger-Virasoro conformal modules than that of Virasoro conformal modules.}
\end{rem}

\section{Extensions of conformal $\mathrm{SV}$-modules}

In this section, we study extensions of conformal $\mathrm{SV}$-modules of the following three types:
\begin{eqnarray}
&&0\longrightarrow \C{c_\gamma}\longrightarrow E_1 \longrightarrow V(\alpha,\Delta) \longrightarrow 0,\label{stype1}\\
&&0\longrightarrow V(\alpha,\Delta)\longrightarrow E_2 \longrightarrow \C c_\gamma \longrightarrow 0,\label{stype2}\\
&&0\longrightarrow V(\bar\alpha,\bar\Delta)\longrightarrow E_3 \longrightarrow V(\alpha,\Delta) \longrightarrow 0, \label{stype3}
\end{eqnarray}
where $V(\alpha,\Delta)=\C[\partial]v_\Delta$ and $V(\bar\alpha,\bar\Delta)=\C[\partial]v_{\bar\Delta}$ are irreducible conformal $\mathrm{SV}$-modules from Proposition \ref{p1}, and $\C{c_\gamma}$ is a one-dimensional $\mathrm{SV}$-module. For each type, we will first describe the trivial extensions and then we classify all the nontrivial cases by giving the formula of the trivial (resp. nontrivial) cocycles.

By Definition \ref{def1}, an $\mathrm{SV}$-module structure on $V$ is given by $L_\lambda, M_\lambda, Y_\lambda \in {\rm End}_\C(V)[\lambda]$ such that
\begin{eqnarray}
&&[L_\lambda, L_\mu]=(\lambda{}-\mu{})L_{\lambda{}+\mu{}},\label{sv1}\\
&&[L_\lambda{},Y_\mu{}]=(\frac12\lambda{}-\mu{})Y_{\lambda{}+\mu{}},\label{sv2}\\
&&[L_\lambda{},M_\mu{}]=-\mu{}M_{\lambda{}+\mu{}},\label{sv3}\\
&&[Y_\lambda{},Y_\mu{}]=(\lambda{}-\mu{})M_{\lambda{}+\mu{}},\label{sv4}\\
&&[\partial{},L_\lambda{}]=-\lambda{}L_\lambda{},\label{sv5}\\
&&[\partial{},Y_\lambda{}]=-\lambda{}Y_\lambda{},\label{sv6}\\
&&[\partial{},M_\lambda{}]=-\lambda{}M_\lambda{},\label{sv7}\\
&&[Y_\lambda{},M_\mu{}]=0,\label{sv8}\\
&&[M_\lambda{},M_\mu{}]=0.\label{sv9}
\end{eqnarray}

We first consider extensions of conformal $\mathrm{SV}$-modules of the form \eqref{stype1}.
 As a module over $\C[\partial]$, $E_1$ in \eqref{stype1} is isomorphic to $\C {c_\gamma}\oplus V(\alpha,\Delta)$, where $\C {c_\gamma}$ is an $\mathrm{SV}$-submodule, and $V(\alpha,\Delta)=\C[\partial]v_\Delta$ such that the following identities hold in $E_1$:
\begin{eqnarray}\label{stype1-1}
L_\lambda{}v_\Delta{}=(\partial{}+\alpha{}+\Delta{}\lambda{})v_\Delta{}+f(\lambda{})c_\gamma{},\ \
M_\lambda{}v_\Delta{}=g(\lambda{})c_\gamma{},\ \
Y_\lambda{}v_\Delta{}=h(\lambda{})c_\gamma{},
\end{eqnarray}
where $f(\lambda)=\sum_{n\geqslant0}f_n\lambda^n,\,g(\lambda)=\sum_{n\geqslant0}g_n\lambda^n$ and $h(\lambda)=\sum_{n\geqslant0}h_n\lambda^n\in \C[\la]$.

The following lemma gives the trivial extensions of the form \eqref{stype1}. We will omit similar calculations in the sequel.
\begin{lemm}\label{lem1} All trivial extensions of the form \eqref{stype1} are of the form \eqref{stype1-1}, where $f(\lambda)$ is a scalar multiple of $\alpha+\gamma+\Delta\lambda$, and $g(\lambda)=h(\lambda)=0$.
\end{lemm}
\begin{proof} Suppose that \eqref{stype1} represents a trivial cocycle. This means that the exact sequence
\eqref{stype1} is split
and hence there exists $v'=\varphi(\partial)v_\Delta+b c_\gamma\in E_1$, where $b\in\C$ and  $0\neq\varphi(\partial)\in\C[\partial]$, such that
\begin{eqnarray*}
L_\lambda v'=(\partial+\alpha+\Delta\lambda)v'
=(\partial+\alpha+\Delta\lambda)\varphi(\partial)v_\Delta+b(\gamma+\alpha+\Delta\lambda)c_\gamma, \ \ Y_\lambda v'=M_\lambda v'=0.
\end{eqnarray*}
On the other hand, it follows from \eqref{stype1-1} that
\begin{eqnarray*}
L_\lambda v'=\varphi(\partial+\lambda)\big((\partial+\alpha+\Delta\lambda)v_\Delta+f(\lambda)c_\gamma\big),\
Y_\lambda v'=\varphi(\partial+\lambda)g(\lambda)c_\gamma,\  M_\lambda v'=\varphi(\partial+\lambda)h(\lambda)c_\gamma.
\end{eqnarray*}
Comparing both expressions for $L_\lambda v'$, $Y_\lambda v'$ and $M_\lambda v'$, respectively, we obtain that $\varphi(\partial)$ is a nonzero constant, $f(\lambda)$ is a scalar multiple of $\alpha+\gamma+\Delta\lambda$, and $g(\lambda)=h(\lambda)=0$.
\end{proof}
\begin{theo}\label{theo1}
Nontrivial extensions of conformal $\mathrm{SV}$-modules of the form \eqref{stype1} exist
if and only if $\alpha+\gamma=0$, and $\Delta=1$, $2$ or $-\frac12$. In these cases, they are given (up to equivalence) by \eqref{stype1-1}, where
\begin{itemize}
\item[{\rm (i)}] $\Delta=1$, $g(\lambda)=h(\lambda)=0$, $f(\lambda)=a_2\lambda^2$, and $a_2\neq0$.
\item[{\rm (ii)}] $\Delta=2$, $g(\lambda)=h(\lambda)=0$, $f(\lambda)=a_3\lambda^3$, and $a_3\neq0$.
\item[{\rm (iii)}] $\Delta=-\frac 12$, $f(\lambda)=g(\lambda)=0$, $h(\lambda)=b_0$, and $b_0\neq0$.
\end{itemize}
In particular, the space ${\rm Ext}(V(\alpha, \Delta),\C{c_{-\alpha}})$ is 1-dimensional in cases (i)--(iii).
\end{theo}
\begin{proof} $``\Leftarrow"$  It follows from Lemma \ref{lem1}.

$``\Rightarrow"$
Suppose that $E_1$ is a nontrivial extension.
Applying both sides of \eqref{sv1}, \eqref{sv2} and \eqref{sv4} to $v_\Delta$ gives
\begin{eqnarray}
&&(\lambda{}-\mu{})f(\lambda+\mu)=(\gamma+\lambda{}+\alpha{}+\Delta{}\mu{})f(\lambda{})-(\gamma+\mu{}+\alpha{}+\Delta{}\lambda{})f(\mu{}),\label{s1}\\
&&(\frac{1}{2}\lambda{}-\mu{})h(\lambda{}+\mu{})=-(\gamma+\mu{}+\alpha{}+\Delta{}\lambda{})h(\mu{}),\label{s2}\\
&&(\lambda{}-\mu{})g(\lambda{}+\mu{})=0.\label{s3}
\end{eqnarray}
Obviously, $g(\lambda)=0$ by \eqref{s3}. Setting $\lambda=0$ in \eqref{s2} gives
\begin{eqnarray}
(\alpha+\gamma)h(\mu)=0. \label{s4}
\end{eqnarray}
If $\alpha+\gamma\neq0$, then $h(\mu)=0$ by \eqref{s4}. Setting $\mu=0$ in \eqref{s1} gives  $f(\lambda)=\frac{f(0)}{\alpha+\gamma}(\alpha+\gamma+\Delta\lambda)$. By Lemma \ref{lem1}, $E_1$ is trivial. A contradiction. Hence $\alpha+\gamma=0.$

Putting $\mu=0$ in \eqref{s2} gives $h(\lambda)=-2\Delta h(0)$. This implies that $h(\lambda)=b_0$ for some $b_0\in\C$. If $b_0\neq0$, then $\Delta =-\frac12$ and it is easy to check that $f(\lambda)=\lambda$ is a unique (up to a scalar) solution to \eqref{s1}.
However, it is a trivial cocycle by Lemma \ref{lem1}. Thus we can assume that $f({\lambda})=0$. By Lemma \ref{lem1} again, the corresponding extension is nontrivial in the case when $b_0\neq0$ and $\Delta =-\frac12$. If $b_0=0$, then it reduces to the Virasoro case and thus the results follow from Theorem \ref{th2}.
\end{proof}

Next we consider extensions of conformal $\mathrm{SV}$-modules of the form \eqref{stype2}.
As a vector space, $E_2$ is isomorphic to $V(\alpha,\Delta)\oplus\C {c_\gamma}$. Here $V(\alpha,\Delta)=\C[\partial]v_\Delta$ is an $\mathrm{SV}$-submodule and the action on $c_\gamma$ is given by
\begin{eqnarray}\label{stype2-1}
L_\lambda{}c_\gamma{}=f(\partial{},\lambda{})v_\Delta{},\ \, M_\lambda{}c_\gamma{}=g(\partial{},\lambda{})v_\Delta{},\ \,
Y_\lambda{}c_\gamma{}=h(\partial{},\lambda{})v_\Delta{},\ \,
\partial{}c_\gamma{}=\gamma{}c_\gamma{}+a(\partial{})v_\Delta,
\end{eqnarray}
where $f(\partial,\lambda),\, g(\partial,\lambda),\, h(\partial{},\lambda{})\in\C[\partial,\lambda]$ and $a(\partial)\in\C[\partial]$.

The following lemma is straightforward.

\begin{lemm}\label{lem2} All trivial extensions of the form \eqref{stype2} are given by \eqref{stype2-1}, where (up to the same scalar) $g(\partial,\lambda)= h(\partial{},\lambda{})=0$,
$f(\partial,\lambda)=(\partial+\alpha+\Delta\lambda)\phi(\partial+\lambda)$ and $a(\partial)= (\partial-\gamma)\phi(\partial)$, with $\phi$ a polynomial.
\end{lemm}

\begin{theo}\label{theo2}
There are nontrivial extensions of $\mathrm{SV}$-modules of the form \eqref{stype2} if and only if $\alpha+\gamma=0$ and $\Delta=1$. In this case, ${\rm dim}_\C\big({\rm Ext}(\C{c_{-\alpha}},V(\alpha,1))\big)=1,$  and the only (up to a scalar) nontrivial extension is given by \eqref{stype2-1},
where $g(\partial,\lambda)=h(\partial,\lambda)=0$ and $f(\partial,\lambda)=a(\partial)=a_0\in \C^*$.
\end{theo}
\begin{proof} By Lemma \ref{lem2}, we only need to prove the necessity. Suppose that $E_2$ in \eqref{stype2} is a nontrivial extension.
Applying both sides of \eqref{sv1}, \eqref{sv5}-\eqref{sv7} to $c_\gamma$ gives the following functional equations:
\begin{eqnarray}
&&(\partial{}+\alpha{}+\Delta{}\lambda{})f(\partial{}+\lambda{},\mu{})-(\partial{}+\alpha{}+\Delta{}\mu{})f(\partial{}+\mu{},\lambda{})=(\lambda{}-\mu{})f(\partial{},\lambda{}+\mu{}),\label{s5}\\
&&(\partial +\lambda-\gamma)f(\partial{},\lambda{})=(\partial{}+\alpha{}+\Delta{}\lambda{})a(\partial{}+\lambda{}),\label{s6}\\
&&(\partial +\lambda-\gamma)h(\partial{},\lambda{})=0,\label{s7}\\
&&(\partial +\lambda-\gamma)g(\partial{},\lambda{})=0.\label{s8}
\end{eqnarray}
We see from \eqref{s7} and \eqref{s8} that  $h(\partial,\lambda)=g(\partial,\lambda)=0$.
Setting $\mu=0$ in \eqref{s5} gives
\begin{eqnarray}
(\partial{}+\alpha{}+\Delta{}\lambda{})f(\partial{}+\lambda{},0)=(\partial{}+\la+\alpha{})f(\partial{},\lambda{}).\label{s9}
\end{eqnarray}
If $(\partial{}+\la+\alpha{})|f(\partial{}+\lambda{},0)$, then set $\phi(\partial+\lambda)= \frac{f(\partial{}+\lambda{},0)}{\partial{}+\la+\alpha{}}$. Hence $f(\partial,\lambda)=(\partial+\alpha+\Delta\lambda)\phi(\partial+\lambda)$ by \eqref{s9}. Plugging this into \eqref{s6}, we obtain $a(\partial{})=(\partial-\gamma)\phi(\partial)$. By Lemma \ref{lem2}, the corresponding extension is trivial. A contradiction.
Therefore $(\partial{}+\la+\alpha{})\not|f(\partial{}+\lambda{},0)$. By \eqref{s9} again, we have $(\partial{}+\la+\alpha{})|(\partial{}+\la+\Delta\alpha)$. It follows that $\Delta=1$ and thus $f(\partial,\la)=f(\partial{}+\lambda{},0)$. Substituting this into \eqref{s6}, we obtain $(\partial{}+\alpha{}+\lambda{})| (\partial +\lambda-\gamma)$. Hence $\alpha=-\gamma$ and $a(\partial)=f(\partial,0)$.

 Suppose that ${\rm deg}(f)\geqslant 1$. In this case, there exist some ${\tilde f}(\pa,\la)\in\C[\pa,\la]$ and $a_0\in\C$ such that $f(\partial,\la)=(\pa+\a+\la){\tilde f}(\pa,\la)+a_0$. Since the polynomial $(\pa+\a+\la){\tilde f}(\pa,\la)$ is divisible by $\pa+\a+\la$, it is a trivial cocycle by the discussions above. Therefore we may always assume that $f(\partial,\la)=a_0$. Then
 $a(\partial)=a_0$. The corresponding extension is nontrivial unless $a_0=0$.
\end{proof}

Finally, we study extensions of conformal $\mathrm{SV}$-modules of the form \eqref{stype3}.
As a $\C[\partial]$-module, $E_3$ in \eqref{stype3} is isomorphic to $V(\bar\alpha,\bar\Delta)\bigoplus V(\alpha,\Delta)$, where $V(\bar\alpha,\bar\Delta)=\C[\partial]v_{\bar\Delta}$ is an  $\mathrm{SV}$-submodule and $V(\alpha,\Delta)=\C[\partial]v_{\Delta}$ with the action of $\mathrm{SV}$ on the vector  $v_\Delta$ given by
\begin{eqnarray}\label{3m*}
L_\lambda v_\Delta=(\partial+\alpha+\Delta\lambda)v_\Delta+f(\partial,\lambda)v_{\bar\Delta },\ \
M_\lambda v_\Delta=g(\partial ,\lambda )v_{\bar\Delta},\ \
Y_\lambda v_\Delta=h(\partial,\lambda)v_{\bar\Delta},
\end{eqnarray}
where $f(\partial,\lambda)=\sum_{n\geqslant0}f_n(\partial)\lambda^n,\ g(\partial,\lambda)=\sum_{n\geqslant0}g_n(\partial)\lambda^n, \ h(\partial,\lambda)=\sum_{n\geqslant0}h_n(\partial)\lambda^n,f_n(\partial),g_n(\partial),h_n(\partial)\in\C[\partial]$ and $f_n(\partial)=g_n(\partial)=h_n(\partial)=0$ for $n\gg0$.

\begin{lemm}\label{lem4} All trivial extensions of the form \eqref{stype3} are given by \eqref{3m*}, where
$f(\partial,\lambda)$ is a scalar multiple of $(\partial+\alpha+\Delta\lambda)\phi(\partial)-(\partial+\bar\alpha+\bar\Delta\lambda)\phi(\partial+\lambda)$, $g(\partial,\lambda)= h(\partial ,\lambda)=0$, and where $\phi$ is a polynomial.
\end{lemm}

Applying both sides of \eqref{sv1}, \eqref{sv2} and \eqref{sv4} to $v_\Delta$ gives the following functional equations:
\begin{eqnarray}
(\lambda-\mu)f(\partial,\lambda+\mu)&=&(\partial+\lambda+\Delta\mu+\alpha)f(\partial,\lambda)+(\partial+\bar\Delta\lambda+\bar\alpha)f(\partial+\lambda,\mu)\nonumber \\
&&-(\partial+\mu+\Delta\lambda+\alpha)f(\partial,\mu)-(\partial+\bar\Delta\mu+\bar\alpha)f(\partial+\mu,\lambda),\label{svl}\\
(\frac12\lambda-\mu)h(\partial,\lambda+\mu)&=&(\partial+\bar\Delta\lambda+\bar\alpha)h(\partial+\lambda,\mu)-(\partial+\mu+\Delta\lambda+\alpha)h(\partial,\mu),\label{svly}\\
(\lambda-\mu)g(\partial,\lambda+\mu)&=&0.\label{svy}
\end{eqnarray}
We see from \eqref{svy} that $g(\partial,\lambda)=0$. The next task is to determine solutions to \eqref{svl} and \eqref{svly}, and to find nontrivial cocycles corresponding to nontrivial extensions of the form \eqref{stype3}.

 If $\alpha\neq\bar\alpha$, then \eqref{svl} and \eqref{svly} with $\lambda=0$ imply that
$f(\partial,\mu)=\frac1{\alpha-\bar\alpha}\big((\partial+\Delta\mu+\alpha)f(\partial,0)-(\partial+\bar\Delta\mu+\bar\alpha)f(\partial+\mu,0)\big)$
and $h(\partial,\mu)=0$. By Lemma \ref{lem4}, the corresponding extension is trivial.

Thus we may assume from now on that $\alpha=\bar\alpha$ for our problem. As in \cite{CKW1}, we put $\bar\partial=\partial+\alpha$ and let $\bar f(\bar\partial,\lambda)=f(\bar\partial-\alpha,\lambda)$, $\bar h(\bar\partial, \lambda)=h(\bar\partial-\alpha,\lambda)$.
 For convenience, we will continue to write $\partial$ for $\bar\partial$, $f$ for $\bar f$, and $h$ for $\bar h$. But keep in mind in the case when $\alpha\neq 0$, we need to perform a shift by $\alpha$ in order to obtain the correct solution. Now we can rewrite \eqref{svl} and \eqref{svly} as follows:
\begin{eqnarray}
(\lambda-\mu)f(\partial,\lambda+\mu)&=&(\partial+\lambda+\Delta\mu)f(\partial,\lambda)+(\partial+\bar\Delta\lambda)f(\partial+\lambda,\mu)\nonumber \\
&&-(\partial+\mu+\Delta\lambda)f(\partial,\mu)-(\partial+\bar\Delta\mu)f(\partial+\mu,\lambda),\label{svl1}\\
(\frac12\lambda-\mu)h(\partial,\lambda+\mu)&=&(\partial+\bar\Delta\lambda)h(\partial+\lambda,\mu)-(\partial+\mu+\Delta\lambda)h(\partial,\mu).\label{svlm1}
\end{eqnarray}
Hence we may assume that $f(\pa,\la)$ and $h(\pa,\la)$ are homogeneous in $\pa$ and $\la$.

Setting $\mu=0$ in \eqref{svlm1} gives
\begin{eqnarray}\label{15-1}
\frac12\lambda h(\partial,\lambda)=(\partial+\bar\Delta\lambda)h(\partial+\lambda,0)-(\partial+\Delta\lambda)h(\partial,0).
\end{eqnarray}
Write $h(\partial,\lambda)=\sum^{m}_{i=0}a_i\partial^{m-i}\lambda^i$. Substituting this into \eqref{15-1}, we have
\begin{eqnarray}\label{coeff}
\sum^{m}_{i=0}\big(\frac{a_i}2\big)\partial^{m-i}\lambda^{i+1}=a_0(\partial+\bar\Delta\lambda)(\partial+\lambda)^m-a_0(\partial+\Delta\lambda)\partial^m.
\end{eqnarray}
If $a_0=0$, then we see from \eqref{coeff} that $a_i=0$ for $i=0, 1,\cdots,m$, and thus $h(\partial,\lambda)=0$. In this case, our problem reduces to the Virasoro case. It has been solved in \cite{CKW1}, and the corresponding results are listed in Theorem \ref{th4}.

In the following we assume that $a_0\neq0$, namely, $h(\partial,\lambda)\neq0$. Comparing the coefficients of $\partial^{m}\lambda$ in \eqref{coeff}, we obtain  $\Delta-\bar\Delta=m-\frac12\in\Z+\frac12$. However, it was shown in \cite{CKW1} that $\Delta-\bar\Delta\in\Z$ if $f(\partial,\lambda)$ is a nonzero solution to \eqref{svl1}. In other words, we always have $f(\partial,\lambda)=0$ in the case when $h(\partial,\lambda)\neq0$. Thus the rest of our task is to solve \eqref{svlm1}.

Equating the coefficients of $\partial^{m-i}\lambda^{i+1}$ in both sides of \eqref{coeff}, we have
\begin{eqnarray}\label{coe}
\frac12a_i=a_0\binom{m}{i+1}+a_0\bar\Delta\binom{m}{i},\ \ 1\leqslant i \leqslant m.
\end{eqnarray}
Plugging $h(\partial,\lambda)=\sum^{m}_{i=0}a_i\partial^{m-i}\lambda^i$ into \eqref{svlm1} gives
\begin{eqnarray}\label{coeff1}
(\frac12\lambda-\mu)\sum^{m}_{i=0}a_i\partial^{m-i}(\lambda+\mu)^i=(\partial+\bar\Delta\lambda)\sum^{m}_{i=0}a_i(\partial+\lambda)^{m-i}\mu^i-(\partial+\mu+\Delta\lambda)\sum^{m}_{i=0}a_i\partial^{m-i}\mu^i.
\end{eqnarray}
Assume that $m\geqslant3$. Comparing the coefficients of $\lambda^m\mu$ in \eqref{coeff1} gives
$(\frac m2-1)a_m=\bar\Delta a_1$ and hence by \eqref{coe}, we obtain
$(\frac m2-1)\bar\Delta=\big(\mbox{$\binom{m}{2}$}+m\bar\Delta\big)\bar\Delta.$
Since $\bar\Delta\neq0$, we have
\begin{eqnarray}\label{coe1}
\frac m2-1=\binom{m}{2}+m\bar\Delta.
\end{eqnarray}
Furthermore, comparing the coefficients of $\partial\lambda^{m-1}\mu$ in \eqref{coeff1} gives
\begin{eqnarray}
\big(\frac {m-1}2-1\big)a_{m-1}=\big(1+(m-1)\bar\Delta\big)a_1.
\end{eqnarray}
Using \eqref{coe} again, we have
\begin{eqnarray}\label{coe2}
\Big(\frac {m-1}2-1\Big)(1+m\bar\Delta)=\big(1+(m-1)\bar\Delta\big)\Big(\binom{m}{2}+m\bar\Delta\Big).
\end{eqnarray}
Combining \eqref{coe1} with \eqref{coe2} implies $\bar\Delta=-\frac12.$ Substituting this back into \eqref{coe1}, we get $m=1$ or $2$. This contradicts the assumption that $m\geqslant 3$.

Therefore, there are only three possible cases for $m$, namely, $ m=0,1,2 $.

Suppose that $m=0$. Then $\Delta-\bar\Delta=-\frac12$ and $h(\partial,\lambda)=$constant is a solution to \eqref{svlm1}.

Next suppose that $m=1$. Then $\Delta-\bar\Delta=\frac12$. It is easily checked that $h(\partial,\lambda)=\partial+2\bar\Delta\lambda$ is a unique, up to a constant, solution to \eqref{svlm1}.

Finally consider the case $m=2$. By \eqref{coe}, we have $a_1=2a_0(1+2\bar\Delta)$ and $a_2=2a_0\bar\Delta$. Thus we may assume that
$h(\partial,\lambda)=\partial^2+(2+4\bar\Delta)\pa\la+2\bar\Delta\lambda^2$.
Plugging this back into \eqref{svlm1} we obtain
\begin{eqnarray*}
&&(\frac12\lambda-\mu)\big(\partial^2+(2+4\bar\Delta)\pa(\la+\mu)+2\bar\Delta(\lambda+\mu)^2\big)\\
&&\ \ \ \ \ \ \ \ \, =(\partial+\bar\Delta\lambda)\big((\partial+\la)^2+(2+4\bar\Delta)(\pa+\la)\mu+2\bar\Delta\mu^2\big)\\
&&\ \ \ \ \ \ \ \ \ \ \ \ \ \,-(\partial+\mu+\Delta\lambda)\big(\partial^2+(2+4\bar\Delta)\pa\mu+2\bar\Delta\mu^2\big).
\end{eqnarray*}
Combined with the fact that $\Delta-\bar\Delta=\frac32$, we can simplify the above equation to $(2+4\bar\Delta)\bar\Delta \la^2\mu=0$. This implies that $\bar\Delta=-\frac12$ since $\bar\Delta\neq 0$. Then it follows that $\Delta=1.$  Therefore, $h(\partial,\lambda)=\partial^2+(2+4\bar\Delta)\pa\la+2\bar\Delta\lambda^2$ is a unique (up to a constant) solution to \eqref{svlm1} in the case when $\Delta=1$ and $\bar\Delta=-\frac12$.

We have proved the following.

\begin{lemm}\label{lemm5} Let $h(\partial,\lambda)=\sum^{m}_{i=0}a_i\partial^{m-i}\lambda^i$ be a
 nontrivial solution to \eqref{svlm1}. Then $a_0\neq 0$, $m=0, 1$ or $2$, and the following is a complete list of values of $\Delta$ and $\bar\Delta$ along with the corresponding polynomials $h(\partial,\lambda)$:

\begin{itemize}
\item[{\rm (i)}] $\Delta-\bar\Delta=-\frac12$, $h(\partial,\lambda)=a_0$;
\item [{\rm (ii)}] $\Delta-\bar\Delta=\frac12$, $h(\partial,\lambda)=a_0(\partial+2\bar\Delta\lambda)$;
\item [{\rm (iii)}] $(\Delta,\ \bar\Delta)=(1,-\frac12)$, $h(\partial,\lambda)=a_0(\partial^2-\lambda^2).$
\end{itemize}
\end{lemm}

By Theorem \ref{th4}, Lemma \ref{lemm5} and the discussions above, we obtain the following results.
\begin{theo}
Nontrivial extensions of conformal $\mathrm{SV}$-modules of the form \eqref{stype3} exist only if $\alpha=\bar\alpha$ and $\Delta-\bar\Delta=0,2,3,4,5,6, \pm\frac12, \frac32.$ The following is a complete list of values of $\Delta$ and $\bar\Delta$, along with the triple of polynomials $f(\partial,\lambda)$, $g(\partial,\lambda)$ and $h(\partial,\lambda)$, whose nonzero scalar multiples give rise to nontrivial extensions ($\bar\partial=\partial+\alpha$):
\begin{itemize}
\item[{\rm (i)}] A nontrivial extension of $\Vir$-modules of Theorem \ref{th4} with trivial actions of $M$ and $Y$.
\item[{\rm (ii)}] $\Delta-\bar\Delta=-\frac12$, $f(\partial,\lambda)=g(\partial,\lambda)=0$ and
$h(\partial,\lambda)=a_0,$ with $a_0\in\C^*$.
\item [{\rm (iii)}]$\Delta-\bar\Delta=\frac12$, $f(\partial,\lambda)=g(\partial,\lambda)=0$
 and $h(\partial,\lambda)=a_0(\bar\partial+2\bar\Delta\lambda),$ with $a_0\in\C^*$.
\item [{\rm (iv)}] $(\Delta,\ \bar\Delta)=(1,-\frac12)$, $f(\partial,\lambda)=g(\partial,\lambda)=0$ and  $h(\partial,\lambda)=a_0(\bar\partial^2-\lambda^2),$ with $a_0\in\C^*$.
\end{itemize}
\end{theo}

\section{Classification of nontrivial irreducible $\widetilde
{\rm SV}$-modules}

This section is devoted to classification of all finite nontrivial irreducible conformal modules over the extended Schr\"odinger-Virasoro conformal algebra $\widetilde
{\rm SV}$.

First let us describe the rank one case, which was given in \cite{SY1}.

\begin{prop}\label{key1}

\begin{itemize}
\item[{\rm (1)}]
All nontrivial conformal $\widetilde {\rm SV}$-modules of rank 1
are the following ($\Delta,\a,\beta\in\C$):
\begin{eqnarray}\label{key-key}
V(\alpha,\beta,\Delta)=\mathbb{C}[\partial]v_\Delta,\ L_\lambda
v_\Delta=(\partial+\alpha+\Delta \lambda)v_\Delta, \ N_\lambda v_\Delta=\beta v_\Delta,\
M_\lambda v_\Delta=Y_\lambda v_\Delta=0.
\end{eqnarray}
\item[{\rm (2)}]
The module
$V(\alpha,\beta,\Delta)$ is irreducible if and only if $\Delta\neq 0$ or $\b\neq 0$.
\end{itemize}
\end{prop}
\begin{proof} (1) It has been proved in \cite{SY1}.

 (2) ``$\Rightarrow$"  Suppose that $\Delta=\b=0$.  In this case, it is easy to see that
 $(\partial+\a)v_\Delta$ generates a proper submodule of $V(\alpha,0,0)$.
Thus  $V(\alpha,0,0)$ is a reducible conformal $\widetilde{\rm SV}$-module. A contradiction. 
Therefore, $(\Delta,\beta)\neq (0, 0)$.

``$\Leftarrow$"  If $\Delta\neq0$, then $V(\alpha,\beta,\Delta)$ is an irreducible conformal Virasoro module.
Thus it is also an irreducible conformal $\widetilde{\rm SV}$-module.

Now consider the case $\b\neq 0$. Suppose that $V$ is a nonzero submodule of $V(\alpha,\beta,\Delta)$. Then
there exists a nonzero element $u=f(\pa)v_\Delta\in V$ for some $0\neq f(\pa)\in\C[\pa]$.
If ${\rm deg}\, f(\pa)=0$, then $v_\Delta\in V$ and thus $V=V(\alpha,\beta,\Delta)$.
If ${\rm deg}\, f(\pa)=m>0$, then $N_{\la}u=f(\pa+\la)(N_{\la}v_\Delta)=\b f(\pa+\la)v_\Delta\in V[\pa,\la]$. We see that
the coefficient of $\la^m$ is a nonzero scalar multiple of $v_\Delta$. This implies that $v_\Delta\in V$ and thus $V=V(\alpha,\beta,\Delta)$.
Therefore, $V(\alpha,\beta,\Delta)$ is irreducible.

This completes the proof.
\end{proof}

\begin{lemm}\label{A}
The annihilation algebra of $\widetilde
{\rm SV}$ is of the form ($m,n\in\Z,$ $p\in \frac12+\Z$)
\begin{eqnarray}
\textit{Lie}(\widetilde
{\rm SV})^+= \sum_{m\geqslant -1}\C L_m+\sum_{n\geqslant 0}\C M_n +\sum_{n\geqslant 0}\C N_n+\sum_{p\geqslant -\frac12}\C Y_{p},
\end{eqnarray}
with the following nonvanishing relations
\begin{eqnarray}\label{LB}\begin{array}{lll}
&&[L_m,L_{n}]=(m-n)L_{m+n},\ \ \
[L_m,M_n]=-nM_{m+n},\ \ \
[L_m,Y_p]=(\frac{m}{2}-p)Y_{m+p}, \\[6pt]
&&[Y_p\,,Y_{q}]=(p-q)M_{p+q},\ \ \ \, [L_m,N_n]=-nN_{m+n}, \ \ \ [N_m,Y_p]=Y_{m+p}, \ \,[N_m,M_n]=2M_{m+n}.\end{array}
\end{eqnarray}
And the extended annihilation algebra $\textit{Lie}(\widetilde
{\rm SV})^e=\C \pa \bigoplus\textit{Lie}(\widetilde
{\rm SV})^+ $, satisfying \eqref{LB} and the following
\begin{equation}\label{LB2}
[\partial, L_m]=-(m+1)L_{m-1}, \
[\partial, M_n]=-nM_{n-1},\ [\partial, N_n]=-nN_{n-1},\ [\partial, Y_p]=-(p+\frac12)Y_{p-1}.
\end{equation}
\end{lemm}
\begin{proof} By \eqref{lamda-bracket} and \eqref{jj1}, we have  $(\pa+2\la)L=[L_\la L]=\sum_{j\geqslant 0}(L_{(j)}L)\frac{\lambda^j}{j!}.$
This implies that
\begin{eqnarray}
L_{(0)}L=\pa L \, \ \ \ L_{(1)}L=2L, \ \ L_{(i)}L=0, \ \ i\geqslant 2.
\end{eqnarray}
Using this along with \eqref{lie}, we obtain
\begin{eqnarray*}
[L_{(m)}, L_{(n)}]&=&\sum_{j\in\Z^+}
\begin{pmatrix}
m\\j
\end{pmatrix}
(L_{(j)}L)_{(m+n-j)}=(L_{(0)}L)_{(m+n)}+m(L_{(1)}L)_{(m+n-1)}\\
&=&-(m+n)L_{(m+n-1)}+2mL_{(m+n-1)}=(m-n)L_{(m+n-1)}.
\end{eqnarray*}
Set $L_{(m)} = L_{m-1}$ for $m\geqslant 0$. Then the above equation is equivalent to $[L_m, L_n]=(m-n) L_{m+n}$. The other relations in \eqref{LB} can be similarly obtained by setting $M_{(n)}=M_n$, $N_{(n)}=N_n$ and $Y_{(n)}=Y_{n-\frac12}$ for $n\geqslant 0$, respectively.
The relations in \eqref{LB2} follows from the fact that $[\pa,a_{(n)}]=-na_{(n-1)}$ for $a_{(n)}\in\textit{Lie}(\widetilde
{\rm SV})^+$.
\end{proof}

The following result is straightforward.
\begin{lemm}\label{center}
$L_{-1}-\pa$ is in the center of $\textit{Lie}(\widetilde {\rm SV})^e$.
\end{lemm}

For simplicity, we denote $\LL=\textit{Lie}(\widetilde {\rm SV})^e$.
Set $\mathcal{L}_n=\bigoplus\limits_{i\geqslant n}(\C L_{i}\oplus\C M_i\oplus\C N_i\oplus\C Y_{i+\frac12})$ for $n \geq 0$ and $\mathcal{L}_{-1}=\C L_{-1}\oplus\C Y_{-\frac12} \oplus \mathcal{L}_0.$
Then we obtain a filtration of the form
\begin{equation*}
\mathcal{L}\supset\mathcal{L}_{-1}\supset\mathcal{L}_0\cdots\supset\mathcal{L}_n\supset\cdots,
\end{equation*}
which satisfies that $\mathcal{L}_{-1}=\textit{Lie}(\widetilde {\rm SV})^+$, $[\mathcal{L}_0, \mathcal{L}_0]=\C M_0+\C Y_{\frac{1}{2}}+\mathcal{L}_1$ and $[\pa,\mathcal{L}_n]=\mathcal{L}_{n-1}$ for $n>0$.

\begin{lemm}\label{E1} For any fixed positive integer $N$,
$\LL_0/\LL_N$ is a finite-dimensional solvable Lie algebra.
\end{lemm}
\begin{proof} It is obvious to see that $\LL_N$ is an ideal of $\LL_0$, and $\LL_0/\LL_N$ is finite-dimensional. To prove solvability of $\LL_0/\LL_N$, we consider the derived subalgebras of $\LL_0$:
$\LL_0^{(0)}=\LL_0$ and $\LL_0^{(n+1)}=[\LL_0^{(n)}, \LL_0^{(n)}]$ for $n\geqslant 0$. By the facts that $\LL_0^{(1)}=[\mathcal{L}_0, \mathcal{L}_0]=\C M_0+\C Y_{\frac{1}{2}}+\mathcal{L}_1$, $[M_{0},\mathcal{L}_1]\subseteq \mathcal{L}_1$ and $[Y_{\frac12},\mathcal{L}_1]\subseteq \mathcal{L}_1$, we obtain $\LL_0^{(2)}\subseteq \mathcal{L}_1$. By induction on $n$, one can obtain that $\LL_0^{(2n)}\subset\LL_n$ for all $n\geqslant 0$. Therefore, for any fixed positive integer $N$, we always have $\LL_0^{(2N)}\subset\LL_N$.
This completes the proof.
\end{proof}

\begin{theo}\label{main} Let $V$ be a nontrivial irreducible conformal module over the extended Schr\"odinger-Virasoro conformal algebra $\widetilde
{\rm SV}$. Then $V\cong V(\alpha,\beta,\Delta)$, which is defined by \eqref{key-key} with $(\Delta,\beta)\neq (0,0)$.
\end{theo}
\begin{proof} By \eqref{conformal}, the conformal $\widetilde
{\rm SV}$-module $V$ can be viewed as a module over $\LL$ satisfying
\begin{eqnarray}
\LL_n v=0, \ \  for \ v\in V,\ \  n\gg 0.
\end{eqnarray}
Set $V_n=\{v\in V\,|\,\LL_n v=0\}$. Then $V_n\neq \{0\}$ for $n\gg 0$.
Let $S$ be the smallest integer such that $V_S\neq\{0\}$. We have $S\geq 0$ (if not, $V$ is a trivial $\LL$-module). By \cite[Lemma 3.1]{CK}, $V_S$ is a finite-dimensional vector space.
  By Lemma \ref{center} and Schur's Lemma, there exists some $\a\in\C$ such that $(L_{-1}-\partial) v=\a v$ for all $v\in V$. Thus
\begin{eqnarray}
L_{-1} v=(\pa+\a) v,\  \forall\ v\in V.
\end{eqnarray}

Suppose that $S=0$. Thus $\LL_0$ acts trivially on $V_0$. Let $u$ be a nonzero element in $V_0$. We claim that $Y_{-\frac12}u=0$. If not, Lemma \ref{I} implies that $w_j=\partial^j (Y_{-\frac12}u)\neq0$, $\forall j\geq 0$. Since $V$ is a free $\C[\partial]$-module of finite rank, there exist $m \in \mathbb{Z}^{+} \text { and } f_{1}(\partial), \ldots, f_{m}(\partial) \in \mathbb{C}[\partial]$ such that
\begin{eqnarray}\label{&&}
\sum_{j=0}^{m} f_{j}(\partial) w_{j}=0
\end{eqnarray}
Denote $n=\max \left\{\operatorname{deg} f_{1}(\partial)+1, \ldots, \operatorname{deg} f_{m}(\partial)+m\right\}$ .
Using $ R_{\partial}=L_{\partial}-\mathrm{ad}_{\partial}$ and the binomial formula, we have, for $k \leq n$,
 \begin{eqnarray}N_{n} \partial^{k}=R_{\partial}^k N_{n}= \left(L_{\partial}-\mathrm{ad}_{\partial}\right)^{k} N_{n}=\sum_{i=0}^{k} \frac{k!n!}{(k-i)!i!(n-i)!} \partial^{k-i} N_{n-i}. \end{eqnarray}
It is easy to see that $N_0(Y_{-\frac12}u)=Y_{-\frac12}u$ and $N_k(Y_{-\frac12}u)=0$ for all $k>0$. By the action of $N_{n}$ on \eqref{&&}, we have $Y_{-\frac12}u=0$. This contradicts the assumption and thus proves the claim.
Now we see that $\C[\pa]u$ is a submodule of $V$. Hence $V=\C[\pa]u$ by the irreducibility of $V$. The $\la$-action of $\widetilde
{\rm SV}$ on $V=\C[\pa]u$ is given by $M_\la u=N_\la u=Y_\la u=0$ and
\begin{eqnarray*}
L_\la u&=&\sum_{j\in\Z^{+}}(L_{(j)}u)\frac{\lambda^j}{j!}=\sum_{j\in\Z^{+}}(L_{j-1} u)\frac{\lambda^j}{j!}=L_{-1} u=(\partial+\alpha)u.
\end{eqnarray*}
By Proposition \ref{key1}, $V$ is reducible, a contradiction.

  Now we have $S\geqslant 1$. By \cite[Lemma 3.1]{CK}, $V_S$ is a finite-dimensional vector space.
Note that $V_S$ is actually an $\LL_0/\LL_S$-module. By Lemma \ref{E1},
there exists a nonzero common eigenvector $u\in V_S$ under the action of $\LL_0/\LL_S$, and then the action of $\LL_0$.
Namely, there exists a linear function $\chi$ on $\LL_0$ such that $x u=\chi(x) u$ for any $x\in\LL_0$. Since $[\LL_0, \LL_0]=\C M_{0}\bigoplus\C Y_{\frac12}\bigoplus \LL_1$, $M_0 u=Y_{\frac12} u=\LL_1 u=0$.
Assume that $\chi(L_0)=\Delta $ and $\chi(N_0)=\beta$ for some $\Delta, \beta\in \C$. It follows $L_0 u=\Delta u$ and $N_0 u=\beta u$. Because $S\geqslant 1$, $(\Delta,\beta)\neq (0,0)$. With a similar discussion as in the case $S=0$, we have $Y_{-\frac12}u=0$.
 Now we can conclude that $\C[\partial]u$ is a nonzero submodule of $V$. Therefore $V=\C[\partial]u$. The $\la$-action of $\widetilde
{\rm SV}$ on $V$ is given by $M_\la u=Y_\la u=0$ and
\begin{eqnarray*}
L_\la u&=&\sum_{j\in\Z^{+}}(L_{(j)}u)\frac{\lambda^j}{j!}=\sum_{j\in\Z^{+}}(L_{j-1} u)\frac{\lambda^j}{j!}=L_{-1} u+(L_0 u)\la=(\partial+\alpha+\Delta\lambda)u,\\
N_\la u&=&\sum_{j\in\Z^{+}}(N_{(j)}u)\frac{\lambda^j}{j!}=\sum_{j\in\Z^{+}}(N_{j} u)\frac{\lambda^j}{j!}=N_0 u=\beta u.
\end{eqnarray*}
 By Proposition \ref{key1}, $V\cong V(\alpha,\beta,\Delta)$ with $(\Delta,\beta)\neq (0,0)$. The proof is finished.
\end{proof}

\section{Extensions of conformal $\widetilde
{\rm SV}$-modules}

In this section, we study extensions of conformal $\widetilde{\mathrm{SV}}$-modules of the following three types:
\begin{eqnarray}
&&0\longrightarrow \C{c_\gamma}\longrightarrow E_1 \longrightarrow V(\alpha,\b,\Delta) \longrightarrow 0,\label{estype1}\\
&&0\longrightarrow V(\alpha,\b,\Delta)\longrightarrow E_2 \longrightarrow \C c_\gamma \longrightarrow 0,\label{estype2}\\
&&0\longrightarrow V(\bar\alpha,\bar\beta,\bar\Delta)\longrightarrow E_3 \longrightarrow V(\alpha,\beta,\Delta) \longrightarrow 0, \label{estype3}
\end{eqnarray}
where $V(\bar\alpha,\bar\beta,\bar\Delta)$ and $V(\alpha,\beta,\Delta)$ are nontrivial irreducible conformal $\widetilde{\mathrm{SV}}$-modules (see Theorem \ref{main}). Thus we will exclude the cases $(\Delta,\b)=(0,0)$ and $(\bar\Delta,\bar\b)=(0,0)$ in the sequel.

By Definition \ref{def1}, an $\widetilde{\mathrm{SV}}$-module structure on $V$ is given by $L_\lambda, M_\lambda, Y_\lambda, N_\lambda \in \rm{End}_\C(V)[\lambda]$ such that \eqref{sv1}--\eqref{sv9}, and the following hold:
\begin{eqnarray}
&&[L_\lambda{},N_\mu{}]=-\mu{}N_{\lambda+\mu},\label{esv1}\\
&&[N_\lambda{},Y_\mu{}]=Y_{\lambda{}+\mu},\label{esv2}\\
&&[N_\lambda{},M_\mu{}]=2M_{\lambda{}+\mu},\label{esv3}\\
&&[N_\lambda{},N_\mu{}]=0,\label{esv4}\\
&&[\partial{},N_\lambda{}]=-\lambda{}N_\lambda.\label{esv5}
\end{eqnarray}

We first consider extensions of $\widetilde{\mathrm{SV}}$-modules of the form \eqref{estype1}. As a $\C[\partial]$-module, $E_1$ in \eqref{estype1} is isomorphic to $\C {c_\gamma}\bigoplus V(\alpha,\b,\Delta)$, where $\C {c_\gamma}$ is an $\widetilde{\mathrm{SV}}$-submodule and $V(\alpha,\b,\Delta)=\C[\partial]v_\Delta$ such that 
\begin{eqnarray}\label{es1}
L_\lambda{}v_\Delta{}=(\partial{}+\alpha{}+\Delta{}\lambda{})v_\Delta{}+f(\lambda{})c_\gamma{},\ \
M_\lambda{}v_\Delta{}=g(\lambda{})c_\gamma{},\ \
Y_\lambda{}v_\Delta{}=h(\lambda{})c_\gamma{}, \ \ N_\lambda{}v_\Delta{}=\b v_\Delta+k(\lambda{})c_\gamma{},
\end{eqnarray}
where $f(\lambda)=\sum_{n\geqslant0}f_n\lambda^n,\,g(\lambda)=\sum_{n\geqslant0}g_n\lambda^n,$ $h(\lambda)=\sum_{n\geqslant0}h_n\lambda^n,$ $k(\lambda)=\sum_{n\geqslant0}k_n\lambda^n\in\C[\la]$.

\begin{lemm}\label{lem6} All trivial extensions of the form \eqref{estype1} are of the form \eqref{es1}, where $g(\lambda)=h(\lambda)=0$, $f(\lambda)=a(\alpha+\gamma+\Delta\lambda)$ and $k(\lambda)=a\b$ for some $a\in\C$.
\end{lemm}

\begin{theo}\label{th5-1}
Nontrivial extensions of $\widetilde{\mathrm{SV}}$-modules of the form \eqref{estype1} exist
if and only if $\alpha+\gamma=0$ and $(\Delta,\beta)=(1,0), (2,0),$ or $ (-\frac12,-1)$. In these cases, they are given (up to equivalence) by \eqref{es1}, where
\begin{itemize}
\item[{\rm (i)}] $(\Delta,\beta)=(1,0)$, $g(\lambda)=h(\lambda)=0$, $k(\lambda)=k_1 \lambda$ and $f(\lambda)=a_2\lambda^2$, with $(k_1, a_2)\neq (0,0)$.
\item[{\rm (ii)}] $(\Delta,\beta)=(2,0)$, $g(\lambda)=h(\lambda)=k(\lambda)=0$ and $f(\lambda)=a_3\lambda^3$, with $a_3\neq0$.
\item[{\rm (iii)}] $(\Delta,\beta)=(-\frac 12,-1)$, $f(\lambda)=g(\lambda)=k(\lambda)=0$ and $h(\lambda)=b_0$, with $b_0\neq0$.
\end{itemize}
In particular,
\begin{eqnarray}
{\rm dim}_\C\big({{\rm Ext}}(V(\alpha, \beta,\Delta),\C{c_{-\alpha}})\big)=\left\{
\begin{array}{ll}
2, &(\Delta,\beta)=(1,0),\\
1, &(\Delta,\beta)=(2,0),\\
1, &(\Delta,\beta)=(-\frac 12,-1).
\end{array}
\right.
\end{eqnarray}
\end{theo}

\begin{proof} By Lemma \ref{lem6}, we only need to prove the necessity. Applying \eqref{esv1}, \eqref{esv2} and \eqref{esv4} to $v_\Delta$ gives
\begin{eqnarray}
&&\beta f(\lambda)-(\alpha+\gamma+\mu+\Delta\lambda)k(\mu)=-\mu k(\lambda+\mu),\label{esv6}\\
&&\beta h(\mu)+h(\lambda+\mu)=0,\label{esv7}\\
&&\beta k(\lambda)-\beta k(\mu)=0.\label{esv9}
\end{eqnarray}

 By \eqref{s3}, $g(\lambda)=0$. If $\alpha+\gamma\neq0,$ then the proof of Theorem \ref{theo1} shows that  $f(\lambda)=\frac{f(0)}{\alpha+\gamma}(\alpha+\gamma+\Delta\lambda)$ and $h(\lambda)=0$. Putting $\lambda=0$ in \eqref{esv6} gives $k(\mu)=\frac{f(0)}{\alpha+\gamma}\b$. Set $a=\frac{f(0)}{\alpha+\gamma}$. It follows from Lemma \ref{lem6} that
the corresponding extension is trivial. Hence $\alpha+\gamma=0.$

Assume that $\beta=0$. It follows from \eqref{esv7} that $h(\lambda)=0$. And \eqref{esv6} amounts to
\begin{eqnarray}
(\mu+\Delta\lambda)k(\mu)=\mu k(\lambda+\mu),\label{esv6-6}
\end{eqnarray}
Setting $\mu=0$ in \eqref{esv6-6} gives $k(0)=0$ since $\Delta\neq0$. Substituting $k(\lambda)=\sum_{n\geqslant0}k_n\lambda^n$ into \eqref{esv6-6} and then comparing the coefficients of $\la^n$, we obtain that $n\leqslant 1$. It is easy to see that  $k(\lambda)=k_1 \lambda$ if $\Delta=1$, or else $k(\lambda)=0$. By \eqref{s1} with $\Delta=1$, we get $f(\lambda)=a_2 \lambda^2$ with $a_2\in\C$. Thus the corresponding extension is nontrivial in the case $(k_1, a_2)\neq (0,0)$. If $\Delta\neq1$, then it follows from \eqref{s1} that $\Delta=2$ and $f(\lambda)=a_3 \lambda^3$ with $a_3\in\C$. The corresponding extension is nontrivial unless $a_3=0$.

Now assume that $\beta\neq0$. By $\eqref{esv9}$ with $\mu=0$, $k(\lambda)=k(0)$ is a constant. Putting $\mu=0$ in \eqref{esv6} gives $f(\lambda)=\frac{k(0)}{\beta}\Delta\lambda$. Set $a=\frac{k(0)}{\beta}$. Then $f(\lambda)=a\Delta\lambda$ and $k(\lambda)=a \beta$. They are trivial cocycles by Lemma \ref{lem6}. Thus we may assume that $k(\lambda)=f(\lambda)=0$. In this case, the extension is nontrivial unless $h(\la)=0$. By \eqref{esv7} with $\mu=0$, $h(\la)=-\beta h(0)$ is a constant. If $h(0)\neq0$, then $\beta=-1$, and $\Delta=-\frac12$ by \eqref{s2}. This completes the proof.
\end{proof}

Next we consider extensions of conformal $\widetilde{\mathrm{SV}}$-modules of the form \eqref{estype2}.
As a vector space, $E_2$ in \eqref{estype2} is isomorphic to $V(\alpha,\b,\Delta)\oplus\C {c_\gamma}$. Here $V(\alpha,\b,\Delta)=\C[\partial]v_\Delta$ is an $\widetilde{\mathrm{SV}}$-submodule and the action on $c_\gamma$ is given by
\begin{eqnarray}\label{2m2}
L_\lambda c_\gamma=f(\partial{},\lambda{})v_\Delta{},\ \, M_\lambda c_\gamma=g(\partial{},\lambda{})v_\Delta{},\ \,
Y_\lambda c_\gamma=h(\partial{},\lambda{})v_\Delta{},\ \, N_\lambda c_\gamma=k(\partial{},\lambda{})v_\Delta{},\ \,
\partial{}c_\gamma=\gamma{}c_\gamma{}+a(\partial{})v_\Delta,
\end{eqnarray}
where $f(\partial,\lambda),\, g(\partial,\lambda),\, h(\partial{},\lambda{}),\,k(\partial,\lambda) \in\C[\partial,\lambda]$ and $a(\partial)\in\C[\partial]$.
\begin{lemm}\label{lem7} All trivial extensions of the form \eqref{estype2} are given by \eqref{2m2} with, up to the same scalar,
$g(\partial,\lambda)= h(\partial{},\lambda{})=0$, $f(\partial,\lambda)= (\alpha+\gamma+\Delta\lambda)\phi(\partial+\lambda)$, $k(\partial,\lambda)=\b \phi(\partial+\lambda)$, and $a(\partial)= (\partial-\gamma)\phi(\partial)$, where $\phi$ is a polynomial.
\end{lemm}

\begin{theo}\label{th5-2}
Nontrivial extensions of $\widetilde{\mathrm{SV}}$-modules of the form \eqref{estype2} exist if and only if $\alpha+\gamma=0$, $\beta=0$ and $\Delta=1$. In this case, ${\rm dim}_\C\big({\rm Ext}(\C{c_{-\alpha}}, V(\alpha,0, 1))\big)=1,$  and the only (up to a scalar) nontrivial extension is given by \eqref{2m2}
with $g(\partial,\lambda)=h(\partial,\lambda)=k(\partial,\lambda)=0$ and $f(\partial,\lambda)=a(\partial)=a_0\in \C^*$.
\end{theo}
\begin{proof} By Lemma \ref{lem7}, we only need to prove the necessity. By the proof of Theorem \ref{theo2}, we have $g(\partial,\lambda)=h(\partial,\lambda)=0$, $\alpha+\gamma=0$, $\Delta=1$ and $f(\partial,\lambda)=a(\partial)=a_0\in \C$. Applying both sides of
\eqref{esv5} to $c_\gamma$ gives
\begin{eqnarray}
(\partial +\lambda-\gamma)k(\partial,\lambda)=\b a(\partial+\lambda).\label{esv11}
\end{eqnarray}
 It follows that $k(\partial,\lambda)=0$ and thus $\b=0$ (If $\b\neq0$, then $a_0=0$ and the corresponding extension is trivial). This completes the proof.
\end{proof}

Finally, we study extensions of conformal $\widetilde{\mathrm{SV}}$-modules of the form \eqref{estype3}.
As a $\C[\partial]$-module, $E_3$ in \eqref{estype3} is isomorphic to $ V(\bar\alpha,\bar\b,\bar\Delta)\bigoplus V(\alpha,\b,\Delta)$, in which $V(\bar\alpha,\bar\b,\bar\Delta)=\C[\partial]v_{\bar\Delta}$ is an $\widetilde{\mathrm{SV}}$-submodule and the action of $\widetilde{\mathrm{SV}}$ on $V(\alpha,\b,\Delta)=\C[\partial]v_{\Delta}$ is given \vspace*{-5pt}by
\begin{eqnarray}\label{estype3**}\begin{array}{lll}
&L_\lambda{}v_\Delta{}=(\partial{}+\alpha{}+\Delta{}\lambda{})v_\Delta{}+f(\partial{},\lambda{})v_{\bar\Delta{}},\ \ \ \ \ \ \ \ &
M_\lambda{}v_\Delta{}=g(\partial{},\lambda{})v_{\bar\Delta{}},\\[6pt]
&N_\lambda{}v_\Delta{}=\beta{}v_\Delta{}+k(\partial{},\lambda{})v_{\bar\Delta{}},\ &Y_\lambda{}v_\Delta{}=h(\partial{},\lambda{})v_{\bar\Delta{}},\
\end{array}
\end{eqnarray}
for some $f(\partial,\lambda), g(\partial, \lambda), h(\partial,\lambda), k(\partial, \lambda)\in\C[\pa,\la]$.

\begin{lemm}\label{lem8} All trivial extensions of the form \eqref{estype3} are given by \eqref{estype3**} with $g(\partial,\lambda)= h(\partial ,\lambda )=0$,
$f(\partial,\lambda)$ and $k(\partial, \lambda)$ of the form  $(\partial+\alpha+\Delta\lambda)\phi(\partial)-(\partial+\bar\alpha+\bar\Delta\lambda)\phi(\partial+\lambda)$, and $\b\phi(\partial)-\bar\b\phi(\partial+\lambda)$, respectively, where $\phi$ is a polynomial.
\end{lemm}

Applying both sides of \eqref{esv1}, \eqref{esv2} and \eqref{esv4} to $v_\Delta$ gives the following functional equations:
\begin{eqnarray}
-\mu{}k(\partial{},\lambda{}+\mu{})&=&\beta f(\partial,\lambda)+(\partial{}+\bar \alpha+\bar \Delta\lambda)k(\partial{}+\lambda{},\mu{})\nonumber\\
&&-(\partial{}+\mu{}+\alpha{}+\Delta{}\lambda{})k(\partial{},\mu{})-\bar\beta{}f(\partial{}+\mu{},\lambda{})
 \label{esv14}\\
h(\partial{},\lambda{}+\mu{})&=&\bar\beta{}h(\partial{}+\lambda{},\mu)-\beta{}h(\partial{},\mu{}), \label{esv15}\\
0&=&\beta{}k(\partial{},\lambda{})+\bar\beta{}k(\partial{}+\lambda{},\mu{})-\beta{}k(\partial{},\mu{})-\bar\beta{}k(\partial{}+\mu{},\lambda{}).\label{esv16}
\end{eqnarray}
By \eqref{svy}, we have $g(\partial,\lambda)=0$. Putting $\lambda=0$ in \eqref{svl} and \eqref{svly} respectively gives that
\begin{eqnarray}
(\alpha-\bar\alpha)f(\partial,\mu)&=&(\partial+\Delta\mu+\alpha)f(\partial,0)-(\partial+\bar\Delta\mu+\bar\alpha)f(\partial+\mu,0),\label{3l1}\\
(\alpha-\bar\alpha)h(\partial,\mu)&=&0.\label{3lh1}
\end{eqnarray}

\begin{lemm}\label{lem9} If $\alpha\neq\bar\alpha$, then there are no nontrivial extensions of conformal $\widetilde{\mathrm{SV}}$-modules of the form \eqref{estype3}.
\end{lemm}
\begin{proof}
Obviously, \eqref{3lh1} gives $h(\partial,\mu)=0$. By \eqref{3l1}, we obtain
\begin{eqnarray}\label{14*}
f(\partial,\mu)=\frac1{\alpha-\bar\alpha}\big((\partial+\Delta\mu+\alpha)f(\partial,0)-(\partial+\bar\Delta\mu+\bar\alpha)f(\partial+\mu,0)\big).
\end{eqnarray}
By \eqref{esv14} with $\lambda=0$, we have
\begin{eqnarray}
k(\partial,\mu)=\frac{1}{\alpha-\bar\alpha}(\beta f(\partial,0)-\bar\b f(\partial+\mu,0).
\end{eqnarray}
Recall that $g(\partial,\lambda)=0$. Set $\phi(\partial)=f(\partial,0)$. By Lemma \ref{lem8}, the corresponding extension is trivial.
\end{proof}

In the following, we always assume that $\alpha=\bar\alpha$.
 As in Section 3, we employ a shift by $\alpha$, which amounts to putting $\bar\partial=\partial+\alpha$, $\bar f(\bar\partial,\lambda)=f(\bar\partial-\alpha,\lambda),$ $\bar h(\bar\partial,\lambda)=h(\bar\partial-\alpha,\lambda)$ and $\bar k(\bar\partial,\lambda)=k(\bar\partial-\alpha,\lambda)$.  For clarity, we will continue to write $\partial$ for $\bar\partial$, $f$ for $\bar f$, $h$ for $\bar h$ and $k$ for $\bar k$. In this case, \eqref{svl1}, \eqref{svlm1}, \eqref{esv15} and \eqref{esv16} still hold, whereas \eqref{esv14} is equivalent to
\begin{eqnarray}
-\mu k(\partial,\lambda+\mu)=\beta f(\partial,\lambda)+(\partial+\bar\Delta\lambda)k(\partial+\lambda,\mu)
-(\partial+\mu+\Delta\lambda)k(\partial,\mu)-\bar\beta f(\partial+\mu,\lambda).\label{esv14-1}
\end{eqnarray}

\begin{lemm}\label{lem10} If $\alpha=\bar\alpha$, $\beta\neq\bar\beta$ and $\bar\beta-\beta\neq1$, then there are no nontrivial extensions of  $\widetilde{\mathrm{SV}}$-modules of the form \eqref{estype3}.
\end{lemm}
\begin{proof} By \eqref{esv15} with $\lambda=0$, we obtain $h(\partial,\mu)=0$ since $\bar\beta-\beta\neq1$. Putting $\mu=0$ in \eqref{esv16} and \eqref{esv14-1} respectively, we have $k(\pa,\lambda)=\frac{\beta k(\pa,0)-\bar\beta k(\pa+\lambda,0)}{\beta-\bar\beta}$ and  $f(\partial,\lambda)=\frac{(\partial+\Delta\lambda)k(\partial,0)-(\partial+\bar\Delta\lambda) k(\partial+\lambda,0)}{\beta-\bar\beta}$. Recall that $g(\partial,\lambda)=0$. By Lemma \ref{lem8}, the corresponding extension is trivial.
\end{proof}

By Lemmas \ref{lem9} and \ref{lem10}, it is left to consider the following three
cases: (1) $\alpha=\bar\alpha$ and $\bar\beta-\beta=1$; (2) $\alpha=\bar\alpha$ and $\beta=\bar\beta=0$; (3) $\alpha=\bar\alpha$ and $\beta=\bar\beta\neq0$.

\begin{theo}\label{lemm-sv3} If $\alpha=\bar\alpha$ and $\bar\beta-\beta=1$, then nontrivial extensions of  $\widetilde{\mathrm{SV}}$-modules of the form \eqref{estype3} exist if and only if $\Delta-\bar\Delta=-\frac12$ or $(\Delta,\bar\Delta)=(\frac{\bar\b+1}2,\frac{\bar\b}2)$. In these cases, they are given (up to equivalence) by \eqref{estype3**}, where ($\bar\pa=\pa+\alpha$)
\begin{itemize}
\item[{\rm (i)}] $\Delta-\bar\Delta=-\frac12$, $f(\pa,\lambda)=g(\pa,\lambda)=k(\pa,\lambda)=0$ and $h(\pa,\lambda)=a_0$ with $a_0\neq 0$.
\item[{\rm (ii)}] $(\Delta,\bar\Delta)=(\frac{\bar\b+1}2,\frac{\bar\b}2)$, $f(\pa,\lambda)=g(\pa,\lambda)=k(\pa,\lambda)=0$ and $h(\pa,\lambda)=a_1(\bar\partial+\bar\beta\lambda)$ with $a_1\neq0$.
\end{itemize}
Furthermore, the space ${\rm Ext}_\C\big(V(\alpha,\beta,\Delta), V(\alpha,\beta+1,\bar\Delta)\big)$ is 1-dimensional in cases (i) and (ii).
\end{theo}
\begin{proof} By Lemma \ref{lem8}, we only need to prove the necessity. Putting $\mu=0$ in \eqref{esv16} and \eqref{esv14-1} along with $\bar\beta-\beta=1$, we have $k(\pa,\lambda)=\bar\beta k(\pa+\lambda,0)-\beta k(\pa,0)$ and $f(\partial,\lambda)=(\partial+\bar\Delta\lambda) k(\partial+\lambda,0)-(\partial+\Delta\lambda)k(\partial,0)$. They are trivial cocycles by Lemma \ref{lem8}. Thus we can assume that $k(\pa,\lambda)=f(\pa,\lambda)=0$. Recall that $g(\partial,\lambda)=0$. It is left to calculate $h(\partial,\lambda)$, which is restricted by \eqref{svlm1} and \eqref{esv15}. Thus we may check if the nontrivial polynomials from Lemma \ref{lemm5} satisfy \eqref{esv15}. Finally, we conclude that
\begin{eqnarray}
h(\partial,\lambda)=\left\{
\begin{array}{ll}
a_0, &\Delta-\bar\Delta=-\frac12, \\
a_0(\partial+\bar\beta\lambda), &(\Delta,\bar\Delta)=(\frac{\bar\b+1}2,\frac{\bar\b}2),\\
0, & otherwise.
\end{array}
\right.
\end{eqnarray}
This completes the proof.
\end{proof}

 \begin{theo}\label{th5-3} If $\alpha=\bar\alpha$ and $\beta=\bar\beta=0$, then nontrivial extensions of  $\widetilde{\mathrm{SV}}$-modules of the form \eqref{estype3} exist only if $\Delta-\bar\Delta=0,1,2,3,4,5,6$. In these cases, they are given (up to equivalence) by \eqref{estype3**}, where $(\bar\pa=\pa+\alpha)$
\begin{itemize}
\item[{\rm (i)}] $\Delta=\bar\Delta$, $g(\pa,\lambda)=h(\pa,\lambda)=0$, $f(\pa,\lambda)=c_0+c_1\lambda$ and $k(\pa,\lambda)=a_0$ with $(c_0,c_1,a_0)\neq (0,0,0)$.
\item[{\rm (ii)}] $\Delta-\bar\Delta=1$, $g(\pa,\lambda)=h(\pa,\lambda)=0$, $f(\pa,\lambda)=0$ and $k(\pa,\lambda)=a_1\lambda$ with $a_1\neq0$.

 \item[{\rm (iii)}]  $\Delta-\bar\Delta=2$, $g(\pa,\lambda)=h(\pa,\lambda)=0$, $f(\pa,\lambda)=c_2\lambda^2(2\bar\partial+\lambda)$ and $k(\pa,\lambda)=a_2\la(\bar\pa-\bar\Delta\la)$ with $(c_2,a_2)\neq (0,0)$.
 \item[{\rm (iv)}] $(\Delta,\bar\Delta)=(1,-2)$, $g(\pa,\lambda)=h(\pa,\lambda)=0$,
 $f(\pa,\lambda)=c_3\bar\pa\lambda^2(\bar\partial+\lambda)$ and $k(\pa,\la)=a_3\la(\bar\pa^2+3\bar\pa\la+2\la^2)$ with $(c_3,a_3)\neq (0,0)$.
 \item[{\rm (iv')}] $\Delta-\bar\Delta=3$, $\bar\Delta\neq -2$,  $g(\pa,\lambda)=h(\pa,\lambda)=k(\pa,\la)=0$ and
 $f(\pa,\lambda)=c_3\bar\pa\lambda^2(\bar\partial+\lambda)$ with $c_3\neq 0$.
 \item[{\rm (v)}] $\Delta-\bar\Delta=4$, $g(\pa,\lambda)=h(\pa,\lambda)=k(\pa,\lambda)=0$ and $f(\pa,\la)=c_5 \la^2(4\bar\pa^3+6\bar\pa^2\la-\bar\pa\la^2+\bar\Delta\la^3)$ with $c_5 \neq 0$.
 \item [{\rm (vi)}]$(\Delta,\bar\Delta)=(1,-4),$ $g(\pa,\lambda)=h(\pa,\lambda)=k(\pa,\lambda)=0$ and  $f(\pa,\la)=c_6(\bar\pa^4\la^2-10\bar\pa^2\la^4-17\bar\pa\la^5-8\la^6)$ with $c_6 \neq 0$.
\item [{\rm (vii)}]$(\Delta,\bar\Delta)=(\frac72\pm\frac{\sqrt{19}}2, -\frac52\pm\frac{\sqrt{19}}2),$ $g(\pa,\lambda)=h(\pa,\lambda)=k(\pa,\lambda)=0$ and $f(\pa,\la)=c_7\big(\bar\pa^4\la^3-(2\bar\Delta+3)\bar\pa^3\la^4-3\bar\Delta\bar\pa^2\la^5
    -(3\bar\Delta+1)\bar\pa\la^6-(\bar\Delta+\frac9{28})\la^7\big)$ with $c_7 \neq 0$.
\end{itemize}
\end{theo}
\begin{proof} By \eqref{esv15}, we see $h(\partial,\mu)=0$ due to $\beta=\bar\beta$. Recall that $g(\partial,\lambda)=0$. It is left to determine $f(\partial,\lambda)$ and $k(\partial,\lambda)$, whose nonzero scalar multiples give rise to nontrivial extension.

Since $\beta=\bar\beta=0$, \eqref{esv14-1} reduces to
\begin{eqnarray}
-\mu k(\partial,\lambda+\mu)=(\partial+\bar\Delta\lambda)k(\partial+\lambda,\mu)
-(\partial+\mu+\Delta\lambda)k(\partial,\mu).\label{esv14-2}
\end{eqnarray}
By the nature of \eqref{esv14-2}, we may assume that a solution to \eqref{esv14-2} is a homogeneous polynomial in $\pa$ and $\la$ of degree $n$.  
Assume that $n\geqslant1$. Write $k(\pa,\la)=\sum_{i=0}^n a_i\pa^{n-i}\la^i$ with $ a_i\in\C$.
 Plugging this into \eqref{esv14-2} gives
\begin{eqnarray}\label{co1}
-\mu\sum^{n}_{i=0}a_i\partial^{n-i}(\lambda+\mu)^i=(\partial+\bar\Delta\lambda)\sum^{n}_{i=0}a_i(\partial+\lambda)^{n-i}\mu^i-(\partial+\mu+\Delta\lambda)\sum^{n}_{i=0}a_i\partial^{n-i}\mu^i.
\end{eqnarray}
Comparing the coefficients of $\la^{n+1}$ in \eqref{co1}, we get $a_0=0$ since $\bar\Delta\neq 0$. This implies that $k(\pa,\la)=\la \tilde k(\pa,\la)$, where $\tilde k(\pa,\la)=\sum_{i=1}^{n} {a_i}\pa^{n-i}\la^{i-1}$. Then \eqref{co1} amounts to
\begin{eqnarray}\label{co2}
(\lambda+\mu)\sum^{n}_{i=1}a_i\partial^{n-i}(\lambda+\mu)^{i-1}=(\partial+\mu+\Delta\lambda)\sum^{n}_{i=1}a_i\partial^{n-i}\mu^{i-1}-(\partial+\bar\Delta\lambda)\sum^{n}_{i=1}a_i(\partial+\lambda)^{n-i}\mu^{i-1}.
\end{eqnarray}
Setting $\mu=0$ in \eqref{co2}, we obtain
\begin{eqnarray}\label{co3}
\sum^{n}_{i=1}a_i\partial^{n-i}\lambda^{i}=(\partial+\Delta\lambda)a_1\partial^{n-1}-(\partial+\bar\Delta\lambda)a_1(\partial+\lambda)^{n-1}.
\end{eqnarray}
We see from \eqref{co3} that, if $a_1=0$, all $a_i=0$ and thus $k(\pa,\la)=0$. A contradiction. Thus $a_1\neq0$. Comparing the coefficients of $\partial^{n-1}\lambda$ in \eqref{co3}, we obtain $\Delta-\bar\Delta=n$. Equating the coefficients of $\partial^{n-i}\lambda^{i}$ in \eqref{co2} gives
\begin{eqnarray}\label{co4}
a_i=-a_1\binom{n-1}{i}-a_1\bar\Delta\binom{n-1}{i-1},\ \ 2\leqslant i \leqslant n.
\end{eqnarray}
Assume that $n\geqslant 4$. Comparing the coefficients of $\lambda^{n-1}\mu$ in \eqref{co2} gives
$n a_n=-\bar\Delta a_2$ and hence by \eqref{co4},
$n\bar\Delta=-\big(\mbox{$\binom{n-1}{2}$}+(n-1)\bar\Delta\big)\bar\Delta.$
Since $\bar\Delta\neq0$, we have
\begin{eqnarray}\label{co5}
\binom{n-1}{2}+(n-1)\bar\Delta+n=0.
\end{eqnarray}
Equating the coefficients of $\partial\lambda^{n-2}\mu$ in \eqref{co2} gives $(n-1)a_{n-1}=-(1+(n-2)\bar\Delta)a_2$. By \eqref{co4} again, 
\begin{eqnarray}\label{co6}
(n-1)(1+(n-1)\bar\Delta)=-\big(1+(n-2)\bar\Delta\big)\Big(\binom{n-1}{2}+(n-1)\bar\Delta\Big).
\end{eqnarray}
Combining \eqref{co5} with \eqref{co6} gives $\bar\Delta=1$. Then \eqref{co5} becomes $\mbox{$\binom{n-1}{2}$}+2n-1=0$, which certainly cannot happen, because $n$ would not be an integer. Therefore $n$ can be at most three, namely, $n=0,1,2, 3$.

We first consider the case $n=3$. By \eqref{co4}, we have $a_2=-(1+2\bar\Delta)a_1$ and $a_3=-a_1\bar\Delta$. Thus we may assume that $k(\pa,\la)=\pa^2\la-(1+2\bar\Delta)\pa\la^2-\bar\Delta\la^3$. Plugging this into \eqref{esv14-2} and combining with the fact that $\Delta-\bar\Delta=3$, we obtain (after simplification)
\begin{eqnarray}
(4\bar\Delta+2\bar\Delta^2)\la^2\mu^2=0.
\end{eqnarray}
It follows $\bar\Delta=-2$. Thus $\Delta=1$ and $k(\pa,\la)=a_1\la(\pa^2+3\pa\la+2\la^2)$ with $a_1\neq 0$.

Similarly, one can easily obtain the following ($a_0\neq 0$, $a_1\neq 0$):
\begin{itemize}
\item[{\rm (1)}]
For $n=0$, $k(\pa,\la)=a_0 $ and $\Delta=\bar\Delta$.
\item[{\rm (2)}]
For $n=1$, $k(\pa,\la)=a_1\la$ and $\Delta-\bar\Delta=1$.
\item[{\rm (3)}]
For $n=2$, $k(\pa,\la)=a_1\la(\pa-\bar\Delta\la)$ and $\Delta-\bar\Delta=2$.
\end{itemize}
Finally, note that the polynomial $f(\pa,\la)$ is completely determined by \eqref{svl1}, which is exactly the Virasoro case. Therefore, combining with Theorem \ref{th4}, we obtain the results.
\end{proof}

\begin{theo}\label{th5-4} If $\alpha=\bar\alpha$ and $\beta=\bar\beta\neq0$, then nontrivial extensions of $\widetilde{\mathrm{SV}}$-modules of the form \eqref{estype3} exist only if $\Delta-\bar\Delta=0,1,2$. In these cases, they are given (up to equivalence) by \eqref{estype3**}, where $(\bar\pa=\pa+\alpha)$
\begin{itemize}
\item[{\rm (i)}] $\Delta=\bar\Delta$, $g(\pa,\lambda)=h(\pa,\lambda)=0$, $k(\pa,\lambda)=a_0$ and $f(\pa,\lambda)=c_0+c_1\lambda$ with $(a_0, c_0,c_1)\neq (0,0,0)$.
\item [{\rm (ii)}] $\Delta-\bar\Delta=1$, $g(\pa,\lambda)=h(\pa,\lambda)=0$,
$k(\pa,\lambda)=b_2\la$ and $f(\pa,\lambda)=c_2\lambda^2$ with $(b_2, c_2)\neq (0,0)$.
\item [{\rm (iii)}] $\Delta-\bar\Delta=2$, $g(\pa,\lambda)=h(\pa,\lambda)=0$, $k(\pa,\la)=b_2\la^2$ and  $f(\pa,\la)=\frac{b_2}{\beta}\bar\pa\lambda^2+c_{3}\la^{3}$ with $(b_2, c_3)\neq (0,0)$.
\end{itemize}
 Moreover, the space ${\rm Ext}_\C\big(V(\alpha,\beta,\Delta), V(\alpha,\beta,\bar\Delta)\big)$ is 3-dimensional in case (i) and 2-dimensional in cases (ii) and (iii).
\end{theo}
\begin{proof} By \eqref{esv15} with $\lambda=0$, we have $h(\partial,\mu)=0$. Recall that $g(\partial,\lambda)=0$. It is left to determine $f(\partial,\lambda)$ and $k(\partial,\lambda)$, whose nonzero scalar multiples give rise to nontrivial extension. Since $\beta=\bar\beta\neq 0$, \eqref{esv16} and \eqref{esv14-1} reduce to
\begin{eqnarray}\label{esv16-1}
0&=&k(\partial{},\lambda{})+k(\partial{}+\lambda{},\mu{})-k(\partial{},\mu{})-k(\partial{}+\mu{},\lambda{}),
\\ \beta \big(f(\partial+\mu,\lambda)-f(\partial,\lambda)\big)&=&\mu k(\partial,\lambda+\mu) +(\partial+\bar\Delta\lambda)k(\partial+\lambda,\mu)
-(\partial+\mu+\Delta\lambda)k(\partial,\mu).\label{esv14-2-1}
\end{eqnarray}
 We can write $k(\pa,\la)=\sum_{i=0}^m a_i\pa^{m-i}\la^i,$ $f(\pa,\la)=\sum_{i=0}^n c_i\pa^{n-i}\la^i,$ where $a_i, c_i\in\C$.

 If $m=0$, then $k(\pa,\la)=a_0$. By \eqref{esv14-2-1}, we have
 \begin{eqnarray}
\beta \big(f(\partial+\mu,\lambda)-f(\partial,\lambda)\big)=a_0(\bar\Delta-\Delta)\lambda,
\label{esv14-3}
\end{eqnarray}
 which implies that ${\rm{deg}} (f) \leqslant 1$ and $f(\pa,0)=0$ when ${\rm{deg}} (f) =1$. Thus $f(\pa,\la)=c_1\la+c_0$. Substituting this into \eqref{svl1} and \eqref{esv14-3}, we obtain $a_0(\bar\Delta-\Delta)=c_0(\bar\Delta-\Delta)=0$. We claim that  $\Delta=\bar\Delta$. If not, $a_0=c_0=0$, namely, $k(\pa,\la)=0$ and $f(\pa,\la)=c_1\la$. By Lemma \ref{lem8}, the corresponding extension is trivial. Therefore, $k(\pa,\la)=a_0$ and $f(\pa,\la)=c_1\la+c_0$ as required. The corresponding extension is nontrivial unless $(a_0,c_0,c_1)=(0,0,0)$.

 In the following we assume that $m\geqslant 1$. Substituting $k(\pa,\la)=\sum_{i=0}^m a_i\pa^{m-i}\la^i$ into \eqref{esv16-1} with $\mu=0$, we have $a_0=0$. Differentiating \eqref{esv16-1} with
respect to $\la$, we obtain
\begin{eqnarray}\label{esv16-2}
0=k_{\la}(\partial{},\lambda{})+k_{\pa}(\partial{}+\lambda{},\mu{})-k_{\la}(\partial{}+\mu{},\lambda{}),
\end{eqnarray}
where $k_{\la}$ and $k_{\pa}$ above denote the partial derivatives of $k(\pa,\la)$ with respect to $\la$ and $\pa$ respectively. Now we put $\la=0$ in \eqref{esv16-2} and get
\begin{eqnarray}\label{esv16-3}
k_{\pa}(\partial{},\mu{})=k_{\la}(\partial{}+\mu{},0)-k_{\la}(\partial{},0)=a_1\big((\pa+\mu)^{m-1}-\pa^{m-1}\big),
\end{eqnarray}
which amounts to $$\sum\limits_{i=1}^{m-1} a_i(m-i)\pa^{m-i-1}\mu^i=a_1\sum\limits_{i=1}^{m-1} \binom{m-1}{i}\pa^{m-i-1}\mu^i.$$
Equating the coefficients of $\pa^{m-i-1}\mu^i$ gives $a_i=\mbox{$\binom{m-1}{i}$}\frac{a_1}{m-i}=\mbox{$\binom{m}{i}$}\frac{a_1}{m}$ for $i=1,\cdots, m-1$. Thus,
\begin{eqnarray}\label{esv16-4}
k(\pa,\la)&=&\frac{a_1}{m}\sum_{i=1}^{m-1} \binom{m}{i}\pa^{m-i}\la^i+a_m\la^m\nonumber\\&=&\frac{a_1}{m}\big((\pa+\la)^m-\pa^m-\la^m\big)+a_m\la^m\nonumber\\
&=&b_1\big((\pa+\la)^m-\pa^m\big)+b_2 \la^m,
\end{eqnarray}
where $b_1=\frac{a_1}{m}$ and $b_2=a_m-\frac{a_1}{m}$. Set $\phi(\partial)=\frac{\pa^m}{\b}$, $k_1(\pa,\la)=(\pa+\la)^m-\pa^m$. Then $k_1(\pa,\la)=\b(\phi(\pa+\la)-\phi(\pa)$, which is a trivial cocycle by Lemma
\ref{lem8}. Thus we may assume that $b_1=0$ in \eqref{esv16-4}. Now we have  $k(\pa,\la)=b_2\la^m$. Substituting this into \eqref{esv14-2-1}, we obtain
\begin{eqnarray}
\beta \big(f(\partial+\mu,\lambda)-f(\partial,\lambda)\big)=b_2\mu (\lambda+\mu)^m +b_2(\partial+\bar\Delta\lambda)\mu^m
-b_2(\partial+\mu+\Delta\lambda)\mu^m.\label{esv14-2-2}
\end{eqnarray}
We see from \eqref{esv14-2-2} that ${\rm{deg}} (f)=m+1$. By \cite[Lemmas 3.1 and 3.2]{CKW1}, we must have $\Delta-\bar\Delta=m$ (if not, we may assume $f(\pa,\la)=0$, then $b_2=0$ by \eqref{esv14-2-2}. Thus $k(\pa,\la)=0$ and the corresponding extension is trivial.) Differentiating \eqref{esv14-2-2} with
respect to $\mu$, we obtain
\begin{eqnarray*}
\beta f_{\pa}(\partial+\mu,\lambda)=b_2\big((\lambda+\mu)^m+m \mu (\lambda+\mu)^{m-1} +m (\partial+\bar\Delta\lambda)\mu^{m-1}-\mu^m
-m(\partial+\mu+\Delta\lambda)\mu^{m-1}\big).
\end{eqnarray*}
Taking $\mu=0$ gives
\begin{eqnarray}
f_{\pa}(\partial,\lambda)=\left\{
\begin{array}{ll}
0, & m=1,\\
\frac{b_2}{\beta}\lambda^m, & m\geqslant 2.
\end{array}
\right.
\end{eqnarray}
Thus (recall that $f(\pa,\la)$ is assumed to be homogenous in $\pa$ and $\la$ of degree $m+1$), we have
\begin{eqnarray}\label{ii}
f(\partial,\lambda)=\left\{
\begin{array}{ll}
c_2\lambda^2, & m=1,\\
\frac{b_2}{\beta}\pa\lambda^m+c_{3}\la^{m+1}, & m\geqslant 2.
\end{array}
\right.
\end{eqnarray}

Assume that $m=1$. By the discussions above, we have $\Delta-\bar\Delta=1$, $k(\pa,\la)=b_2 \la$ and $f(\pa,\la)=c_2\la^2$.
The corresponding extension is nontrivial in the case $(b_2,c_2)\neq (0,0)$.

Now assume that $m=2$. In this case, $\Delta-\bar\Delta=2$,  $k(\pa,\la)=b_2\la^2$ and  $f(\pa,\la)=\frac{b_2}{\beta}\pa\lambda^2+c_{3}\la^{3}$. One can check that they are unique (up to the same scalar) solutions to \eqref{svl1}, \eqref{esv16-1} and \eqref{esv14-2}. And it is easy to see that the corresponding extension is nontrivial in the case $(b_2,c_3)\neq (0,0)$.

Finally consider the case $m\geqslant3$.
By \eqref{ii}, $f(\pa,\la)=\frac{b_2}{\beta}\pa\lambda^m+c_{3}\la^{m+1}$. Plugging this into \eqref{esv14-2-2}, we obtain (after simplification)
\begin{eqnarray}
b_2\la^m=b_2(\la+\mu)^m+b_2\bar\Delta\la\mu^{m-1}-b_2\mu^m-b_2\Delta\la\mu^{m-1}.
\end{eqnarray}
Comparing the coefficients of $\la^2\mu^{m-2}$ in the equation above, we get $\mbox{$\binom{m}{2}$}b_2=0$ and thus $b_2=0$. Then $k(\pa,\la)=0$ and $f(\pa,\la)=c_3\la^{m+1}$. Substituting this into \eqref{svl1} and then comparing the coefficients of $\la^2\mu^{m}$, we obtain
$\Big((m+1)-\mbox{$\binom{m+1}{2}$}\Big)c_3=0$, leading to $c_3=0$. Thus $f(\pa,\la)=0$, and the extension is trivial. This completes the proof.
\end{proof}

\section {Applications to the Heisenberg-Virasoro conformal algebra}

 The Heisenberg-Virasoro conformal algebra, denoted by $\mathrm{HV}$, was introduced in \cite{SY1} as a subalgebra of the
extended Schr\"odinger-Virasoro Lie conformal algebra. Recall that $\mathrm{HV}=\C[\pa]L\bigoplus \C[\pa]N$
with the $\lambda$-brackets defined by \eqref{elie}. It was shown in \cite{WY} that every finite irreducible conformal $\mathrm{HV}$-module is either $V(\alpha,\beta,\Delta)=\mathbb{C}[\partial]v_\Delta$, with actions defined by
\begin{eqnarray}\label{key-key111}
L_\lambda
v_\Delta=(\partial+\alpha+\Delta \lambda)v_\Delta, \ \ \ N_\lambda v_\Delta=\beta v_\Delta,
\end{eqnarray}
where $\a,\b,\Delta\in\C$ with $(\Delta,\b)\neq (0,0)$, or else it is the one-dimensional module $\C c_\gamma$ with $L_\la c_\gamma=N_\la c_\gamma=0$, and  $\pa c_\gamma=\gamma c_\gamma$, $\gamma\in\C.$
Comparing \eqref{key-key111} with \eqref{key-key}, we see that a finite nontrivial irreducible
conformal $\widetilde{\mathrm{SV}}$-module is just an irreducible
conformal $\mathrm{HV}$-module with trivial actions of $M$ and $Y.$ This allows us to use the classification of extensions of the extended Schr\"odinger-Virasoro conformal modules obtained in Section 5 to solve the extension problem for the Heisenberg-Virasoro conformal algebra.

Consider the following three types of extensions between two finite irreducible conformal modules over the Heisenberg-Virasoro conformal algebra:
\begin{eqnarray}
&&0\longrightarrow \C{c_\gamma}\longrightarrow E_1 \longrightarrow V(\alpha,\b,\Delta) \longrightarrow 0,\label{hvtype1}\\
&&0\longrightarrow V(\alpha,\b,\Delta)\longrightarrow E_2 \longrightarrow \C c_\gamma \longrightarrow 0,\label{hvtype2}\\
&&0\longrightarrow V(\bar\alpha,\bar\beta,\bar\Delta)\longrightarrow E_3 \longrightarrow V(\alpha,\beta,\Delta) \longrightarrow 0, \label{hvtype3}
\end{eqnarray}
where $(\Delta,\b)\neq(0,0)$ and $(\bar\Delta,\bar\b)\neq(0,0)$.

As a $\C[\partial]$-module, $E_1$ in \eqref{hvtype1} is isomorphic to $\C {c_\gamma}\bigoplus V(\alpha,\b,\Delta)$, where $\C {c_\gamma}$ is an ${\mathrm{HV}}$-submodule and $V(\alpha,\b,\Delta)=\C[\partial]v_\Delta$ with
\begin{eqnarray}\label{hv1}
L_\lambda v_\Delta=(\partial{}+\alpha+\Delta\lambda)v_\Delta+f(\lambda)c_\gamma,\ \ N_\lambda v_\Delta=\b v_\Delta+k(\lambda)c_\gamma,  \ \ f(\lambda),\ k(\lambda)\in\C[\la].
\end{eqnarray}

By Lemma \ref{lem6} and  Theorem \ref{th5-1}, we have

\begin{coro}\label{coro1} Nontrivial extensions of ${\mathrm{HV}}$-modules of the form \eqref{hvtype1} exist if and only if $\alpha+\gamma=0,$ $\beta=0$ and $\Delta=1$ or $2$. They are given, up to equivalence, by \eqref{hv1}, where
\begin{itemize}
\item[{\rm (i)}] $k(\lambda)=k_1\lambda$, $f(\lambda)=f_2\lambda^2$, for $\Delta=1$ and $(k_1,f_2)\neq (0,0)$.
\item[{\rm (ii)}]$k(\lambda)=0$, $f(\lambda)=f_3\lambda^3$, for $\Delta=2$ and $f_3\neq0$.
\end{itemize}
Furthermore, all trivial cocycles are given by the same scalar multiples of the polynomials $f(\lambda)=\alpha+\gamma+\Delta\lambda$ and $k(\lambda)=\b$.
\end{coro}

As a vector space, $E_2$ in \eqref{hvtype2} is isomorphic to $V(\alpha,\b,\Delta)\oplus\C {c_\gamma}$. Here $V(\alpha,\b,\Delta)=\C[\partial]v_\Delta$ is an ${\mathrm{HV}}$-submodule and we have
\begin{eqnarray}\label{hv2}
L_\lambda c_\gamma=f(\partial,\lambda)v_\Delta,\ N_\lambda c_\gamma=k(\partial,\lambda)v_\Delta{},\ \,
\partial{}c_\gamma=\gamma{}c_\gamma{}+a(\partial{})v_\Delta,
\end{eqnarray}
where $f(\partial,\lambda),\,k(\partial,\lambda)\in\C[\partial,\lambda]$ and $a(\partial)\in\C[\partial]$.

By Lemma \ref{lem7} and Theorem \ref{th5-2}, we have

\begin{coro}\label{coro2}
Nontrivial extensions of $\mathrm{HV}$-modules of the form \eqref{hvtype2} exist if and only if $\alpha+\gamma=0$, $\beta=0$ and $\Delta=1$. The unique nontrivial extension is given, up to equivalence, by \eqref{hv2}
with $k(\partial,\lambda)=0$, $f(\partial,\lambda)=a(\partial)=a_0,$ $a_0\in \C^*$. Furthermore, all trivial extensions correspond to the triples of the form $f(\partial,\lambda)= (\alpha+\gamma+\Delta\lambda)\phi(\partial+\lambda)$, $k(\partial,\lambda)=\b \phi(\partial+\lambda)$, and $a(\partial)= (\partial-\gamma)\phi(\partial)$, where $\phi$ is a polynomial.
\end{coro}

Let $E_3$ be an extension of the form \eqref{hvtype3}. As a $\C[\partial]$-module, $E_3$ in \eqref{hvtype3} is isomorphic to $ \C[\partial]v_{\bar\Delta}\bigoplus \C[\partial]v_{\Delta}$, where $\C[\partial]v_{\bar\Delta}$ is an $\mathrm{HV}$-submodule and the action of $\mathrm{HV}$ on $\C[\partial]v_{\Delta}$ is given by
\begin{eqnarray}\label{hv3}
L_\lambda{}v_\Delta{}=(\partial{}+\alpha{}+\Delta{}\lambda{})v_\Delta{}+f(\partial{},\lambda{})v_{\bar\Delta{}},\
N_\lambda{}v_\Delta{}=\beta{}v_\Delta{}+k(\partial{},\lambda{})v_{\bar\Delta{}},\
\end{eqnarray}
for some polynomials $f(\partial,\lambda)$ and $k(\partial, \lambda)$.

By Lemma \ref{lem8} and Theorems \ref{th5-3}--\ref{th5-4}, we have
\begin{coro}\label{coro3}
 Nontrivial extensions of $\mathrm{HV}$-modules of the form \eqref{hvtype2} exist only if $\alpha=\bar\alpha$ and $\beta=\bar\beta$. These extensions are given, up to equivalence, by \eqref{hv3},
where the values of $\Delta$ and $\bar\Delta$ along with the corresponding polynomials $f(\pa,\la)$ and $k(\pa,\la)$ are listed as follows $(\bar\pa=\pa+\alpha)$:
\begin{itemize}
\item For $\b=\bar\beta=0$:
\begin{itemize}
\item[{\rm (i)}] $\Delta=\bar\Delta$, $f(\pa,\lambda)=c_0+c_1\lambda$, $k(\pa,\lambda)=a_0$ with $(c_0,c_1,a_0)\neq (0,0,0)$.
\item[{\rm (ii)}] $\Delta-\bar\Delta=1$, $f(\pa,\lambda)=0$, $k(\pa,\lambda)=a_1\lambda$ with $a_1\neq0$.

 \item[{\rm (iii)}]  $\Delta-\bar\Delta=2$, $f(\pa,\lambda)=c_2\lambda^2(2\bar\partial+\lambda)$, $k(\pa,\lambda)=a_2\la(\bar\pa-\bar\Delta\la)$ with $(c_2,a_2)\neq (0,0)$.
 \item[{\rm (iv)}] $(\Delta,\bar\Delta)=(1,-2)$,
 $f(\pa,\lambda)=c_3\bar\pa\lambda^2(\bar\partial+\lambda)$, $k(\pa,\la)=a_3\la(\bar\pa^2+3\bar\pa\la+2\la^2)$ with $(c_3,a_3)\neq (0,0)$.
 \item[{\rm (iv')}] $\Delta-\bar\Delta=3$, $\bar\Delta\neq-2$, $k(\pa,\la)=0$,
 $f(\pa,\lambda)=c_3\bar\pa\lambda^2(\bar\partial+\lambda)$ with $c_3\neq 0$.
 \item[{\rm (v)}] $\Delta-\bar\Delta=4$, $k(\pa,\lambda)=0$, $f(\pa,\la)=c_5 \la^2(4\bar\pa^3+6\bar\pa^2\la-\bar\pa\la^2+\bar\Delta\la^3)$ with $c_5 \neq 0$.
 \item [{\rm (vi)}]$(\Delta,\bar\Delta)=(1,-4),$ $k(\pa,\lambda)=0$, $f(\pa,\la)=c_6(\bar\pa^4\la^2-10\bar\pa^2\la^4-17\bar\pa\la^5-8\la^6)$ with $c_6 \neq 0$.
\item [{\rm (vii)}]$(\Delta,\bar\Delta)=(\frac72\pm\frac{\sqrt{19}}2, -\frac52\pm\frac{\sqrt{19}}2),$ $k(\pa,\lambda)=0$, $f(\pa,\la)=c_7\big(\bar\pa^4\la^3-(2\bar\Delta+3)\bar\pa^3\la^4-3\bar\Delta\bar\pa^2\la^5
    -(3\bar\Delta+1)\bar\pa\la^6-(\bar\Delta+\frac9{28})\la^7\big)$ with $c_7 \neq 0$.
\end{itemize}
\item For $\beta=\bar\beta\neq 0$:
\begin{itemize}
\item[{\rm (i)}] $\Delta=\bar\Delta$, $k(\pa,\lambda)=a_0$, $f(\pa,\lambda)=c_0+c_1\lambda$ with $(a_0, c_0,c_1)\neq (0,0,0)$.
\item [{\rm (ii)}] $\Delta-\bar\Delta=1$,
$k(\pa,\lambda)=b_2\la$, $f(\pa,\lambda)=c_2\lambda^2$ with $(b_2, c_2)\neq (0,0)$.
\item [{\rm (iii)}] $\Delta-\bar\Delta=2$, $k(\pa,\la)=b_2\la^2$,  $f(\pa,\la)=\frac{b_2}{\beta}\bar\pa\lambda^2+c_{3}\la^{3}$ with $(b_2, c_3)\neq (0,0)$.
\end{itemize}
\end{itemize}
Furthermore, all trivial extensions correspond to pairs of the form $f(\partial,\lambda)=(\partial+\alpha+\Delta\lambda)\phi(\partial)-(\partial+\bar\alpha+\bar\Delta\lambda)\phi(\partial+\lambda)$ and $k(\partial,\lambda)=\b\phi(\partial)-\bar\b\phi(\partial+\lambda)$, where $\phi$ is a polynomial.
\end{coro}

\begin{rem} We studied extensions of Heisenber-Virasoro conformal modules in \cite{LY}. However, there exist some mistakes in the main results. Thus we correct the mistakes by Corollaries \ref{coro1}--\ref{coro3} in this paper.
\end{rem}

\noindent\bf{ Acknowledgements.}\ \rm
{\footnotesize This work was supported by National Natural Science
Foundation of China (11301109), the Research Fund for the Doctoral Program of Higher Education (20132302120042) and China Scholarship Council.}

\end{document}